\documentclass[11pt]{article}%
\usepackage{float} % fix table inside context 
\usepackage{fullpage}
\usepackage{amsfonts}
\usepackage{amsmath}
\usepackage{amssymb}
\usepackage{amsbsy}
\usepackage{amsthm}
\usepackage{mathtools}
\usepackage{epsfig}
\usepackage{graphicx}
\usepackage{colordvi}
\usepackage{graphics}
\usepackage{color}
\usepackage{bm}
\usepackage{verbatim} 
\numberwithin{equation}{section}

\newtheorem{theorem}{Theorem}[section]
\newtheorem{thm}{Theorem}[section]

\newtheorem{lemma}{Lemma}[section]

\newtheorem{alg}[thm]{Algorithm}
\newtheorem{remark}[thm]{Remark}

\usepackage{algpseudocode}
\newcommand{\f}{\frac}
\DeclareMathOperator{\argmin}{argmin}

\newcommand{\be}{\begin{equation}}
\newcommand{\ee}{\end{equation}}
\newcommand{\bea}{\begin{eqnarray}}
\newcommand{\eea}{\end{eqnarray}}
\newcommand{\beas}{\begin{eqnarray*}}
	\newcommand{\eeas}{\end{eqnarray*}}

\newcommand{\vertiii}[1]{{\left\vert\kern-0.25ex\left\vert\kern-0.25ex\left\vert #1
		\right\vert\kern-0.25ex\right\vert\kern-0.25ex\right\vert}}

\newcommand{\norm}[1]{\lVert#1\rVert}
\newcommand{\normiii}[1]{{\left\vert\kern-0.25ex\left\vert\kern-0.25ex\left\vert #1
		\right\vert\kern-0.25ex\right\vert\kern-0.25ex\right\vert}}

\begin{document}
\title{Accelerating and enabling convergence of nonlinear solvers for Navier-Stokes equations by continuous data assimilation}

\author{Xuejian Li\thanks{\small
		School of Mathematical and Statistical Sciences, Clemson University, Clemson, SC, 29364, xuejial@clemson.edu.}  \and Elizabeth V. Hawkins\thanks{\small
	School of Mathematical and Statistical Sciences, Clemson University, Clemson, SC, 29364, evhawki@clemson.edu.} \and	Leo G. Rebholz\thanks{\small
		School of Mathematical and Statistical Sciences, Clemson University, Clemson, SC, 29364, rebholz@clemson.edu.}
	\and		Duygu Vargun\thanks{\small
		School of Mathematical and Statistical Sciences, Clemson University, Clemson, SC, 29364, dvargun@clemson.edu.}	}

	\maketitle
	
	\begin{abstract}{
		This paper considers improving the Picard and Newton iterative solvers for the Navier-Stokes equations in the setting where data measurements or solution observations are available.  We construct adapted iterations that use continuous data assimilation (CDA) style nudging to incorporate the known solution data into the solvers.  For CDA-Picard, we prove the method has an improved convergence rate compared to usual Picard, and the rate improves as more measurement data is incorporated.  We also prove that CDA-Picard is contractive for larger Reynolds numbers than usual Picard, and the more measurement data that is incorporated the larger the Reynolds number can be with CDA-Picard still being contractive.  For CDA-Newton, we prove that the domain of convergence, with respect to both the initial guess and the Reynolds number, increases as the amount of measurement data is increased.
Additionally, for both methods we show that CDA can be implemented as direct enforcement of measurement data into the solution.  Numerical results for common benchmark Navier-Stokes tests illustrate the theory.
	}
\end{abstract}

	\section{Introduction}
	
	The Navier-Stokes equations (NSE) are widely used to model incompressible Newtonian fluid flow, which has applications across the spectrum of science and engineering.  We consider herein the steady state NSE, which are given on a domain $\Omega \subset \mathbb{R}^d, d=2,3$ by
	\begin{equation}\label{NS}
		\left\{\begin{aligned}
			-\nu \Delta u+u\cdot\nabla u+ \nabla p&={f} \quad \text{in}~\Omega,\\
			\nabla\cdot {u}&=0\quad \text{in}~\Omega,\\
			{u}&=0 \quad \text{on}~~\partial\Omega,
		\end{aligned}\right.
	\end{equation}
	where $u$ is the velocity of fluid, $p$ is the pressure, $\nu$ is the kinematic viscosity of the fluid,  and ${f}$ is an external forcing term. The parameter $Re:=\frac{1}{\nu}$ represents the Reynolds number.   While our study herein is restricted to nonlinear solvers for the steady system \eqref{NS} with homogenous Dirichlet boundary conditions, the results are extendable to solving the time dependent NSE at a fixed time step in a temporal discretization, as well as nonhomogeneous mixed Dirichlet/Neumann boundary conditions. 
	
	Perhaps the most common nonlinear iteration for solving \eqref{NS}  is  the Picard iteration, which is given by (suppressing the boundary conditions for the moment)
	\begin{align*}
		-\nu \Delta u_{k+1}+u_k\cdot\nabla u_{k+1}+ \nabla p_{k+1}&={f}, \\
		\nabla\cdot {u}_{k+1}&=0.
	\end{align*}
	The Picard iteration is known to be globally convergent under a smallness condition on the Reynolds number / problem parameters $\alpha:=M\nu^{-2} \|f \|_{H^{-1}}<1$, where $M$ is a domain-size dependent constant arising from Sobolev inequalities (it is defined precisely in section 2) \cite{GR86,PRX19}.  Provided $\alpha<1$, the constant $\alpha$ is also an upper bound on Picard's $H^1$ linear convergence rate \cite{GR86} and is believed reasonably sharp (although calculating $\alpha$ in practice is typically imprecise).  
	
	While the Newton iteration is also popular for solving \eqref{NS} due to its quadratic convergence, Newton requires a good initial guess and is often used together with the Picard iteration.  A common strategy is using Picard until sufficiently close to the solution and then switching to Newton.  The Newton iteration for the NSE is given by {\color{red} (\cite{J16} p. 339)}
	\begin{align*}
		-\nu \Delta u_{k+1}+u_k\cdot\nabla u_{k+1}+ u_{k+1}\cdot\nabla u_k + \nabla p_{k+1} &={f} + u_k \cdot\nabla u_k, \\
		\nabla\cdot {u}_{k+1}&=0.
	\end{align*}

	The purpose of this paper is to improve both the Picard and Newton iterations for the NSE in the setting where solution data is available, e.g. from measurements or observed data.  We construct a new algorithm which uses CDA style nudging to incorporate data into the solution via the following CDA-Picard iteration
	\begin{align*}
		-\nu \Delta u_{k+1}+u_k\cdot\nabla u_{k+1}+ \nabla p_{k+1} + \mu I_H(u_{k+1} - u)&={f}, \\
		\nabla\cdot {u}_{k+1}&=0,
	\end{align*}
	and CDA-Newton iteration
	\begin{align*}
		-\nu \Delta u_{k+1}+u_k\cdot\nabla u_{k+1}+ u_{k+1}\cdot\nabla u_k + \nabla p_{k+1} + \mu I_H(u_{k+1} - u)&={f} + u_k \cdot\nabla u_k, \\
		\nabla\cdot {u}_{k+1}&=0.
	\end{align*}
	Here, $\mu>0$ is a user defined nudging parameter and $I_H$ is an appropriate interpolant on a mesh with maximum element width $H$.  Note that the true solution $u$ remains unknown, but $I_H(u)$ is known from given or observed solution data. 
	
	The use of CDA above is based on the Azouani, Olson, and Titi algorithm from 2014 \cite{AOT14}, and we note
	its usual intent is for time dependent systems.  It has been used on a wide variety of problems such as NSE and turbulence \cite{AOT14,DMB20,FLT19,CGJP22}, Benard convection \cite{FJT15}, planetary geostrophic modeling \cite{FLT16}, the Cahn-Hilliard equation \cite{DR22}, and many others.  CDA has gained substantial interest in the last decade, which has led to many improvements and uses for it, including for parameter recovery \cite{CHL20,CH22}, sensitivity analyses \cite{CL21}, and numerical methods and analyses \cite{IMT20,LRZ19,RZ21,DR22,GNT18,JP23}.  To our knowledge, this paper is the first time CDA is used to improve efficiency and robustness of nonlinear solvers for steady PDEs.
	
	We prove that with a sufficient amount of solution data (e.g. from measurements or observables), CDA-Picard will improve the convergence rate compared to Picard: we show that in a weighted $H^1$ norm, the convergence rate of CDA-Picard is $\sim \sqrt{H}\alpha$; recall for usual Picard it is $\alpha$ in the $H^1$ norm \cite{GR86,Laytonbook}.  This result not only implies that the more solution data that is available the faster the convergence will be, but additionally proves the the CDA-Picard method will {\color{red} be globally convergent} provided $\sqrt{H} \alpha<O(1)$; this is made precise in our analysis.  We also prove that CDA helps Newton, and prove that CDA-Newton has a much larger domain of convergence with respect to both the initial guess and $\alpha$, and the domain of convergence increases further as more solutions data is available. 
	
	Additionally, another important result of the paper is that the methods used to incorporate the data are incredibly simple to implement.  Although our analysis holds for any $\mu>0$, we show from our analysis that $\mu=\infty$ is a good parameter choice (i.e. there are no ill-effects of having very large $\mu$), and {\color{red} leads to the decoupling of the mathematical condition/constraint} $I_H u_{k+1} = I_H u$.  This {\color{red} can in turn lead to} $u_{k+1}$ taking on the measurement values directly, and an algorithm where no nudging is explicitly used but instead data measurements are incorporated directly into the solution in the same way that Dirichlet boundary conditions are implemented.

In some sense, the results of our analysis are not particularly surprising from a high level - that is, it should be expected that nonlinear iterations can be improved if we know some of the solution data.  The novelty of this paper is that we show specifically {\it how} to incorporate the data efficiently and are able to prove precise mathematical results for the convergence rates, sizes of variables that will allow for convergence, and how much measurement data is needed to guarantee these results.

	The paper is organized as follows. In Section \ref{Pre}, we provide notation and mathematical preliminaries. 
	In Section \ref{CDA-Picard}, we analyze the convergence of the CDA-Picard iteration and use numerical tests to illustrate the theory. In Section \ref{CDA-Newton},  we focus on the convergence analysis and numerical tests for CDA-Newton. Finally, we draw conclusions in Section \ref{conclusion}.

	\section{Preliminaries}\label{Pre}
	
	We consider an open connected set $\Omega$ that is either a convex polygon or has smooth boundary.  We denote the natural function spaces for the NSE by 
	\begin{align}
		&Q:=\{v\in {L}^2(\Omega): \int_{\Omega}vdx=0\},\\
		& X:=
		\{v\in H^1\left(\Omega\right): v=0~~\text{on}~ \partial\Omega\},\\
		& V:=
		\{v\in X:  (\nabla \cdot v,q)=0\ \forall q\in Q\}.
	\end{align}
	We use $(\cdot,\cdot)$ to denote the $L^2$ inner product that induces the $L^2$ norm $\|\cdot\|$, $\langle \cdot,\cdot\rangle$ to denote the duality between $H^{-1}$ and $X$, and $\|\cdot\|_{-1}$ to denote the operator norm $\|\cdot\|_{\langle H^{-1},X\rangle}$. 
	
	We recall the Poincare inequality holds on $X$: there exists a constant $C_P$ dependent only on the domain satisfying $\| \phi \| \le C_P \|\nabla \phi \|$ for all $\phi\in X$.  This inequality implies the $H^1$ norm is equivalent to the $L^2$ gradient norm, and we will use the latter as the norm for $X$ and $V$.

\subsection{Steady NSE preliminaries}
	
	The weak form of the NSE (\ref{NS}) is to find $(u,p)\in X\times Q$ such that 
	\begin{equation}\label{weakNavier}
		\left\{\begin{aligned}
				a\left({u},v\right)+b\left({u},{u},v\right)+(p,\nabla\cdot v)&=\left\langle{f},v\right\rangle~~\forall v\in X,\\
			(\nabla\cdot{u},q)&=0~~\forall q\in Q,
		\end{aligned}\right.
	\end{equation}
	%\begin{equation}
	%	\label{weakNavier}
	%	a\left({u},v\right)+b\left({u},{u},v\right)+(p,\nabla\cdot v)=\left\langle{f},v\right\rangle~~\forall v\in X,~~(\nabla\cdot{u},q)=0~~\forall q\in Q,
	%\end{equation}
	where $a(\cdot,\cdot)$ and $b(\cdot,\cdot,\cdot)$ are defined as 
	\begin{align*}
		a(u,v)&=\left(\nu\nabla u, \nabla v\right)~~\forall u,v\in X\\b(u,w,v)&=\frac12 \left((u\cdot\nabla)w, v\right) - \frac12 \left((u\cdot\nabla)v, w\right) ~~\forall u,w,v\in X.
	\end{align*}
Since the pair $(X,Q)$ satisfies the inf-sup condition, we can instead consider the equivalent system \cite{GR86}: Find $u\in V$ satisfying
	\begin{equation}\label{wd}
		a\left({u},v\right)+b\left({u},{u},v\right)=\left\langle{f},v\right\rangle~~\forall v\in V.
	\end{equation}

While we use the skew-symmetric form of the nonlinear term, we note the same analysis and results hold below with other energy preserving formulations of the nonlinear term including rotational form and EMAC \cite{CHOR17,OR20}.	The key point is that $b(u,v,v)=0$ for all $u,v\in X$.
The following inequalities hold for $b$ \cite{Laytonbook,temam}: there exists a constant $M$ dependent only on the domain $\Omega$ such that
	\begin{align}
		b(u,w,v)&\leq M\|u\|^{\frac{1}{2}}\|\nabla u\|^{\frac{1}{2}}\|\nabla w\|\|\nabla v\|,\\
b(u,w,v)&\leq M \|\nabla u\|\|\nabla w\|\|\nabla v\|.
	\end{align}

	We recall the classical  well-posedness result for equation (\ref{wd}) \cite{Laytonbook,temam}:
	
	\begin{lemma}
		Let $\alpha=M\nu^{-2}\|f\|_{-1}$. For any $f\in H^{-1}$ and $\nu$,	there exists at least one solution for NSE (\ref{wd}), the solution has a priori estimate 
		\begin{align}\label{Pri}
			\|\nabla u\|\leq \nu^{-1}\|f\|_{-1}.
		\end{align}
		Furthermore, if $\alpha<1$, 
		the solution is unique.
	\end{lemma}
	The restriction $\alpha< 1$ is usually referred to as the small data condition for steady NSE, but we refer to it herein as a smallness condition since we use data to mean observations or measurements (and not given PDE parameters).

\subsection{Picard and Newton iterations}

The Picard and Newton iterations for the steady NSE can be written in their $V$-formulations as follows: Find $u_{k+1}\in V$ satisfying for all $v\in V$ that
		\begin{align}
				\label{PicardO}a\left({u}_{k+1},v\right)+b\left({u}_{k},{u}_{k+1},v\right)&=\left\langle{f},v\right\rangle~~ \text{(Picard)},\\
				\label{NewtonO}a\left({u}_{k+1},v\right)+b\left({u}_{k},{u}_{k+1},v\right)+b\left({u}_{k+1},{u}_{k},v\right)& =\left\langle{f},v\right\rangle +b\left({u}_{k},{u}_{k},v\right)~~ \text{(Newton)}.
		\end{align}
	
	To help provide insight into the CDA-Picard and CDA-Newton methods studied below, we now recall some results for Picard (\ref{PicardO}) and Newton (\ref{NewtonO}).  First we give the Picard convergence result.  This is known from e.g. \cite{GR86}, and we include a proof in the Appendix for completeness.

		\begin{lemma}\label{PLemma}
		The Picard method (\ref{PicardO}) is unconditionally stable. Furthermore, if $\alpha<1$
		then the sequence $\{u_k\}$ generated by Picard converges to the NSE solution $u$ as $k\to \infty$ with an $\alpha$-linear rate.
	\end{lemma}

Next, we give a result from Newton, similar to what is in \cite{GR86} or classical literature.  For completeness, we provide a proof in the Appendix.

	\begin{lemma}\label{nc}
		Let $\beta=\nu(1-2\alpha)$. Assume $8\alpha\leq 1$ and let $u$ denote the NSE solution.  If the initial guess satisfies		
		$
		\frac{M}{\beta}\|\nabla (u - u_0)\|<1,
		$
		then the Newton method (\ref{NewtonO}) converges to $u$   (\ref{wd}) quadratically:
		\begin{align}
			\|\nabla (u - u_{k+1})\|\leq \frac{M}{\beta}\|\nabla (u - u_k)\|^2.
		\end{align}
	\end{lemma}

\subsection{Interpolation and CDA implementation}

We now define the CDA iterative solvers.  We begin with the interpolation operator.  Denote by $\tau_H(\Omega)$ a coarse mesh of $\Omega$ {\color{red} used to represent an interpolant of the true solution using measurements or observables.
%We choose $I_H = P_1(\tau_H)$, the continuous piecewise linear interpolant on $\tau_H$, although for our implementation discussed below in this subsection, there are many %different choices of $I_H$ that would have the exact same implementation.  
Assumed} properties of $I_H$ are that it satisfies the usual CDA estimates: there exists a constant $C_I$ independent of $H$ satisfying
		\begin{align}
			\label{interpolationi}	\|I_Hv-v\|\leq C_IH \| \nabla v\|~~~~\forall v\in X,\\
			\label{interpolationi2}	\|I_Hv\|\leq C_I \|v\|~~~~\forall v\in X. 
		\end{align}
{\color{red}Some examples of such interpolation operators include Bernardi-Girault \cite{BG98}, Scott–Zhang \cite{SZ90}, and the $L^2$ projection onto piecewise constants \cite{fem:book:ern:guermond}.  
%Algebraic nudging also satisfies these properties \cite{RZ21}, however the constant in \eqref{interpolationi} may have a slight negative dependence on the fine mesh width $h$ in 2D and $h^{-1/2}$ in 3D \cite{RZ21}.  
}

Although our analysis is general to $(X,Q)$, in practice one uses finite dimensional (finite element) subspaces $(X_h,Q_h)$ defined on a fine mesh $\tau_h$.  For the implementation details discussed here and performed in numerical tests to follow, we require that nodes of $\tau_H$ also be nodes of $\tau_h$, and $h \ll H$.

The CDA-Picard and CDA-Newton iterations are given in their $V$-formulations by
\begin{align}
			\label{Picard}	&a\left({u}_{k+1},v\right)+b\left({u}_{k},{u}_{k+1},v\right)+\mu (I_Hu_{k+1}-I_H  u,I_H v)=\left\langle{f},v\right\rangle~~\text{(CDA-Picard)}\\
			\label{Newton}	 	&a\left({u}_{k+1},v\right)+b\left({u}_{k},{u}_{k+1},v\right)+b\left({u}_{k+1},{u}_{k},v\right)+\mu (I_Hu_{k+1}-I_H u,I_H v)\nonumber\\& \ \ \ \ \ =\left\langle{f},v\right\rangle+b\left({u}_{k},{u}_{k},v\right)~~~~ \text{(CDA-Newton)}.
		\end{align}
Here, $\mu I_H(u_{k+1}-u)$ is a CDA nudging driving the state $u_{k+1}$ towards to the observations, and $\mu$ is a positive relaxation parameter that emphasizes the observations' accuracy.  We note that we use a type of variational crime on the nudging term, by using $I_H$ in the second argument of the nudging terms.  While this is only consistent with the PDE if $I_H$ is an $L^2$ projection, the use of $I_H$ additionally on the test function is key to allowing for less restrictions on parameters and no upper bound on $\mu$.  This nudging term / variational crime was first used in \cite{RZ21} to construct algebraic nudging, and was shown in \cite{GN20} to help reduce restrictions in CDA.

As will be revealed in the next sections, the convergence results we find for CDA-Picard and CDA-Newton require only that $\mu$ be sufficiently large, i.e. $\mu>O(H^{-2})$.  That is, there is no negative consequence of taking $\mu$ even larger, even to $\mu=\infty$.  While taking $\mu=\infty$ in practice could lead to numerical roundoff error, when put into the algorithms (before implementation) {\color{red} it decouples the constraint that $I_H u_{k+1} = I_H u$.  Provided the measurement locations are also nodes on the fine mesh, this can be implemented just like a Dirichlet boundary condition and without any explicit nudging at all, since we assume no noise in the measurement data}.  Hence CDA-Picard with $\mu=\infty$ {\color{red} takes the form} simply of usual Picard \eqref{PicardO} but with additional constraints $(u^h_{k+1})_j = u(x_j)$ for each of the $j=1,2,...$ measurement nodes (and analogous for direct implementation in CDA-Newton).  In the case where data measurements have noise, this approach and CDA in general can still be used until the residual falls to the same order of magnitude as the noise.

\subsection{Anderson acceleration}

In our numerical tests in section 3, we will apply Anderson acceleration (AA) to CDA-Picard.  We define AA here, following \cite{PR23}.  We assume a fixed point function $g:X\rightarrow X$, and $y_k:= g(x_{k-1}) - x_{k-1}$ is the $k^{th}$ step residual.  For our iterations, $g$ can be considered the solution operator of a single iteration of (CDA-)Picard.

\begin{alg}[Anderson acceleration]\label{alg:aa}
%The algorithm begins with a fixed point update step at $k=0$,
%and then the acceleration starts at $k=1$.
\ \\
\begin{algorithmic}
\State Choose initial iterate $x_0$, and algorithmic depth parameter $m$
\State Compute $w_1$, %set $m_0 = 0$, 
set relaxation parameter $\beta_0$, and update $x_1 = x_0 + \beta_0 w_1$
\Comment{$k=0$}
\end{algorithmic}
\begin{algorithmic}[1]
\For{$k = 1, \ldots$}
\Comment{$k > 0$}
\State Compute $y_{k+1}$
%\State Update $m_k = \min\{k,m\}$
\State Set $m_k = \min\{k,m\}$
\State 
Set {\color{black}$F_k= \begin{pmatrix}(y_{k+1}-y_k) & \hdots & 
                         (w_{k+1-(m_k-1)} - w_{k-(m_k-1)})\end{pmatrix}$}\\
\hspace{.4in} and  {\color{black}$E_k= \begin{pmatrix}(e_{k}) & \ldots &
                         (e_{k-(m_k-1)})\end{pmatrix}$}
\State
Find $\gamma_k = \argmin  \|F_k \gamma - y_{k+1}\|_X$
\State Set relaxation parameter $\beta_k$
\State
Update $x_{k+1} = x_k + \beta_k y_{k+1} - \left(E_k + \beta_k F_k \right)\gamma_{k+1}$
\EndFor
\end{algorithmic}
\end{alg}

We interpret the minimization as a weighted
least-squares problem, which can be efficiently solved using for instance a QR 
factorization as in \cite{WaNi11}. 

Define the optimization gain $\theta_{k+1}$ by
\begin{align}\label{eqn:gaindef}
\theta_{k+1} \coloneqq \f{\norm{F_k \gamma_k - y_{k+1}}_X}{\norm{y_{k+1}}_X}.
\end{align}
As shown in \cite{EPRX20,PR21,PRX19}, at each iteration $k$, AA improves the linear convergence rate by a factor of $\theta_k$, but at the cost of adding additional higher-order terms.

\section{The CDA-Picard iteration}\label{CDA-Picard}

In this section we analyze the CDA-Picard iteration, which is given in weak form by: Given NSE solution data $I_H u$ and $u_k\in V$, find $u_{k+1}\in V$ satisfying
\begin{equation}\label{CDAPicard}
	a\left({u_{k+1}},v\right)+b\left({u_k},{u_{k+1}},v\right) + \mu(I_H(u_{k+1} - u),I_H v)=\left\langle{f},v\right\rangle~~\forall v\in V.
\end{equation}

We give below analysis and numerical tests that show how CDA can accelerate and even enable convergence in the Picard iteration.

\subsection{Analysis of CDA-Picard}

We now prove a convergence result for CDA-Picard.  The convergence is in the weighted $H^1$ norm
	\begin{align}\label{d1}
		\|v\|_{*}=\left(\|\nabla v\|^2+\frac{1}{2C_I^2H^2}\| v\|^2\right)^{\frac{1}{2}}.
\end{align}
While the convergence analysis for the usual Picard method is in the $H^1$ norm, for CDA-Picard the analysis is done in this {\color{red} weighted $H^1$ norm, as this appears to be the natural norm from the analysis.  If $H$ is very small, the $*$-norm becomes close to the $L^2$ norm.  Still, our results show} that CDA both accelerates and enables convergence, even though the comparison of the analysis is not quite apples to apples.  Of course once spatially discretized, all norms are equivalent and so convergence in any norm means convergence in all norms.

\vspace{1mm}
\begin{theorem}\label{Pcac}
Let $u$ be a steady NSE solution and suppose $I_H(u)$ is known.  If $\sqrt{2}C_IH\alpha^2<1$ and $\mu\geq \frac{\nu}{4C_I^2H^2}$, then for any initial guess $u_0$ the CDA-Picard iteration \eqref{CDAPicard} converges to $u$ linearly with rate (at least) $\sqrt{{\color{red}\sqrt{2}}C_IH}\alpha$:
	\begin{align}\label{CDAPC}
		\|u - u_{k+1}\|_{*}&\leq \sqrt{{\color{red}\sqrt{2}}C_IH}\alpha\|u - u_{k}\|_{*}.
	\end{align}
\end{theorem}

\begin{remark}
	Theorem \ref{Pcac} shows that convergence is guaranteed for any $\alpha$, provided enough data is available (i.e. $H$ is small enough).  It also shows that the convergence rate $O(H^{1/2}\alpha)$ is improved by including data measurements.
\end{remark}

\begin{proof}
Let $e_k=u-u_k$.
	Subtracting \eqref{wd}  from \eqref{CDAPicard} and then rearranging terms gives us,
	\begin{equation}\label{sb}
		\begin{split}
			0&=a(u,v)-a(u_{k+1},v)+b(u,u,v)-b(u_k,u_{k+1},v)-\mu (I_Hu_{k+1}-I_H u,I_Hv)\\
			%	&=a(u-u_{k+1},v)+b(u,u,v)-b(u_k,u,v)+b(u_k,u,v)-b(u_k,u_{k+1},v)\\&~~~~+\mu (I_Hu-I_Hu_{k+1},I_Hv)\\
			&=a(u-u_{k+1},v)+b(u-u_k,u,v)+b(u_k,u-u_{k+1},v)+\mu (I_Hu-I_Hu_{k+1},I_Hv)\\
			&=a(e_{k+1},v)+b(e_k,u,v)+b(u_k,e_{k+1},v)+\mu (I_He_{k+1},I_Hv).
		\end{split}
	\end{equation}
	Taking $v=e_{k+1}$ in equation (\ref{sb}), applying the bounds on $b$, Young's inequality, and the inequality (\ref{Pri}), we deduce an upper bound for the term $\nu\|\nabla e_{k+1}\|^2+\mu\|I_He_{k+1}\|^2$:
	\begin{equation}\label{upper}
		\begin{split}
			\nu\|\nabla e_{k+1}\|^2+\mu\|I_He_{k+1}\|^2&=a(e_{k+1},e_{k+1})+\mu (I_He_{k+1},I_He_{k+1})
			\\&=-b(e_k,u,e_{k+1})\\
			&\leq M\|e_k\|^{\frac{1}{2}}\|\nabla e_k\|^{\frac{1}{2}}\|\nabla u\| \|\nabla e_{k+1}\|\\
			&\leq \alpha\nu\|e_k\|^{\frac{1}{2}}\|\nabla e_k\|^{\frac{1}{2}} \|\nabla e_{k+1}\|\\
			&\leq \frac{\nu}{2}\|\nabla e_{k+1}\|^2+\frac{\nu\alpha^2}{2}\| e_{k}\|\|\nabla e_{k}\|\\
			&\leq \frac{\nu}{2}\|\nabla e_{k+1}\|^2+B\nu\|\nabla e_{k}\|^2+\frac{\nu\alpha^4}{16B}\| e_{k}\|^2, 
		\end{split}
	\end{equation}
	where $B>0$ is arbitrary for now.	
	Using inequality (\ref{interpolationi}) and the triangle inequality, we bound the term $\nu\|\nabla e_{k+1}\|^2+\mu\|I_He_{k+1}\|^2$ from below by
	\begin{equation}\label{lowerb}
		\begin{split}
			\nu  \|\nabla e_{k+1}\|^2+\mu\|I_He_{k+1}\|^2 & =\frac{\nu}{4}\|\nabla e_{k+1}\|^2+\mu\|I_He_{k+1}\|^2+\frac{3\nu}{4}\|\nabla e_{k+1}\|^2\\
			&\geq\frac{\nu}{4C_I^2H^2}\|e_{k+1}-I_He_{k+1}\|^2+\mu\|I_He_{k+1}\|^2+\frac{3\nu}{4}\|\nabla e_{k+1}\|^2\\
			&\geq \lambda \left(\|e_{k+1}-I_He_{k+1}\|^2+\|I_He_{k+1}\|^2\right)+\frac{3\nu}{4}\|\nabla e_{k+1}\|^2\\
			&\geq \frac{\lambda}{2} \|e_{k+1}\|^2+\frac{3\nu}{4}\|\nabla e_{k+1}\|^2,
		\end{split}
	\end{equation}
	where $\lambda=\min\{\mu,\frac{\nu}{4C_I^2H^2}\} = \frac{\nu}{4C_I^2H^2}$ by assumption on $\mu$ being sufficiently large. 
	Combining (\ref{upper}) and (\ref{lowerb}) and then reducing leads to
	\begin{equation}\label{lower}
		\begin{split}
			\|\nabla e_{k+1}\|^2+\frac{1}{2C_I^2 H^2} \|e_{k+1}\|^2&\leq			4 B\left(\| \nabla e_{k}\|^2+\frac{\alpha^4}{16B^2}\| e_{k}\|^2\right).
		\end{split}
	\end{equation}
	Choosing $B = \sqrt{ \frac{\alpha^4C_I^2 H^2}{8} }$, we get  
		$\frac{1}{2C_I^2 H^2}= \frac{\alpha^4}{16B^2}$
	and so
{\color{red}	{	\begin{align}\label{PN}
			\|\nabla e_{k+1}\|^2+\frac{1}{2C_I^2 H^2} \|e_{k+1}\|^2&\leq			\sqrt{2}\alpha^2 C_I H \left(\| \nabla e_{k}\|^2+\frac{1}{2C_I^2 H^2}\| e_{k}\|^2\right).
	\end{align}}}
Now utilizing the norm $\|\cdot\|_*$, we have the reduced estimate
\begin{align}
	\|e_{k+1}\|_{*}&\leq \sqrt{{\color{red}\sqrt{2}} C_IH}\alpha\|e_{k}\|_{*}.
\end{align}
The proof is now completed.  
\end{proof}
\vspace{2mm}

\subsection{Numerical tests for CDA-Picard}\label{PCTest}

We now illustrate the above theory, and show how incorporating measurement data via CDA can accelerate and even enable convergence in the Picard iteration for the steady NSE.  In all of our tests, we take $u_0=0$ and {\color{red} take as $I_H$ the $L^2$ projection onto $X_H=P_0(\tau_H)$ and implement it  by using a first order quadrature approximation, i.e. with algebraic nudging \cite{RZ21}.  While our analysis does not consider the discretization of the nonlinear iterations (although it is analogous and gives analogous results), this implementation can lead in 3D (but not 2D) to a larger constant $C_I$ from \eqref{interpolationi} that may depend on $h^{-1/2}$ \cite{RZ21}.  For these first proof of concept tests below (where $h$ is fixed and not too small) we believe the ease in implementation is worth the potentially increased constant, however practitioners working on 3D problems with limited measurement data may need to consider higher order quadrature approximation for the interpolation operator. }

\subsubsection{Scaling of the linear convergence rate with $H$}

For our first test, we calculate the scaling with $H$ of the *-norm convergence, using the benchmark 2D driven cavity problem on the unit square $\Omega=(0,1)^2$.  We use no external forcing $f=0$ and Dirichlet boundary condition, that is, no-slip on the sides and bottom, and $\langle 1,0\rangle^T$ on the lid.  In this problem $Re:=\nu^{-1}$.  For this test of the scaling with $H$, we choose $Re=100$ and compute with $(P_2,P_1)$ elements on a uniform $\frac{1}{64}$ triangulation, varying $H$ and direct enforcement of the measurement data.  A convergence plot is shown in figure \ref{CP}, and we observe an improvement in the convergence rate as $H$ is decreased. Table \ref{tableP} shows the calculated linear convergence rates in the $*$-norm, which are consistent with a $H^{1/2}$ scaling although it is not clearly 0.5 for the exponent.   This scaling result is asymptotic, and already at $H=$1/64 we have reached $H=h$.  We note other test problems were tried as well as finer meshes, and similar results were found for scaling with $H$.

\begin{figure}[H]
\centering
\includegraphics[height=5.5cm,viewport=0 0 420 290, clip]{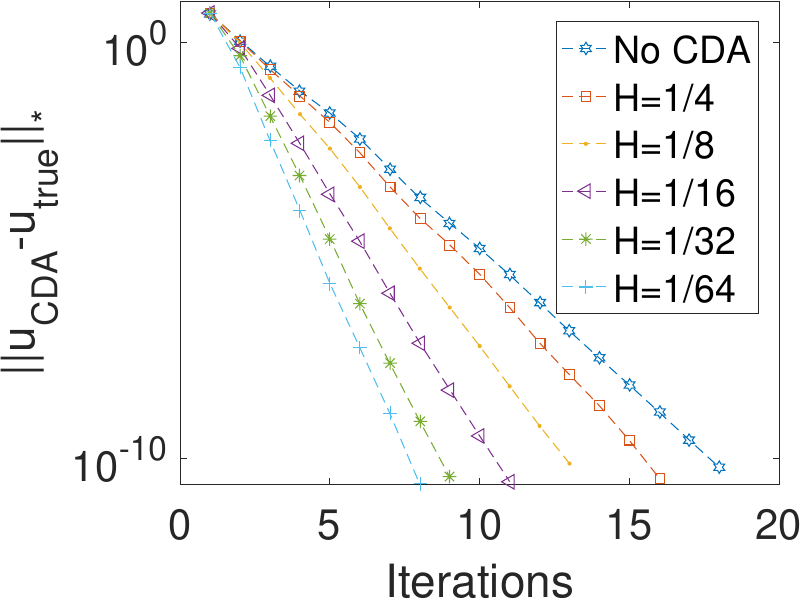}
\caption{Shown above are residuals in the CDA-Picard convergence tests}\label{CP}
\end{figure}

\begin{table}[H]
\centering
\begin{tabular}{|c|c|c|c|}
	\hline
	H &  {\# iterations}&{linear conv rate in $*$-norm}&{scaling of linear rate wrt $H$}\\ \hline
1/4 &      16      &      0.1814  &              ----\\ \hline
1/8  &     13     &       0.1211    &            0.5829\\ \hline
1/16 &     11   &         0.0705    &            0.7805\\ \hline
1/32  &    9   &          0.0371      &          0.9260\\ \hline
1/64   &   8 &            0.0231       &         0.6835\\ \hline
\end{tabular}
\caption{Shown above are linear convergence rates for CDA-Picard with varying $H$, and their calculated scaling with $H$.
}\label{tableP}
\end{table}

\subsubsection{2D driven cavity}

For our next test we again consider CDA-Picard applied to the benchmark 2D driven cavity problem on the unit square $\Omega=(0,1)^2$.  We now consider $Re$=3,000,\ 5,000, and 10,000 and $(P_2,P_1^{disc})$ elements on a mesh of $\Omega$ constructed as a barycenter refinement (also called Alfeld split in the Guzman-Neilan vernacular) of a barycenter refined uniform $\frac{1}{128}$ triangulation; it is known from \cite{arnold:qin:scott:vogelius:2D} that this element choice on this type of structured mesh is inf-sup stable.

\begin{figure}[ht]
\center
Re=3,000 \hspace{1in} Re=5,000 \hspace{1in} Re=10,000 \\
\includegraphics[width = .27\textwidth, height=.28\textwidth,viewport=115 45 465 390, clip]{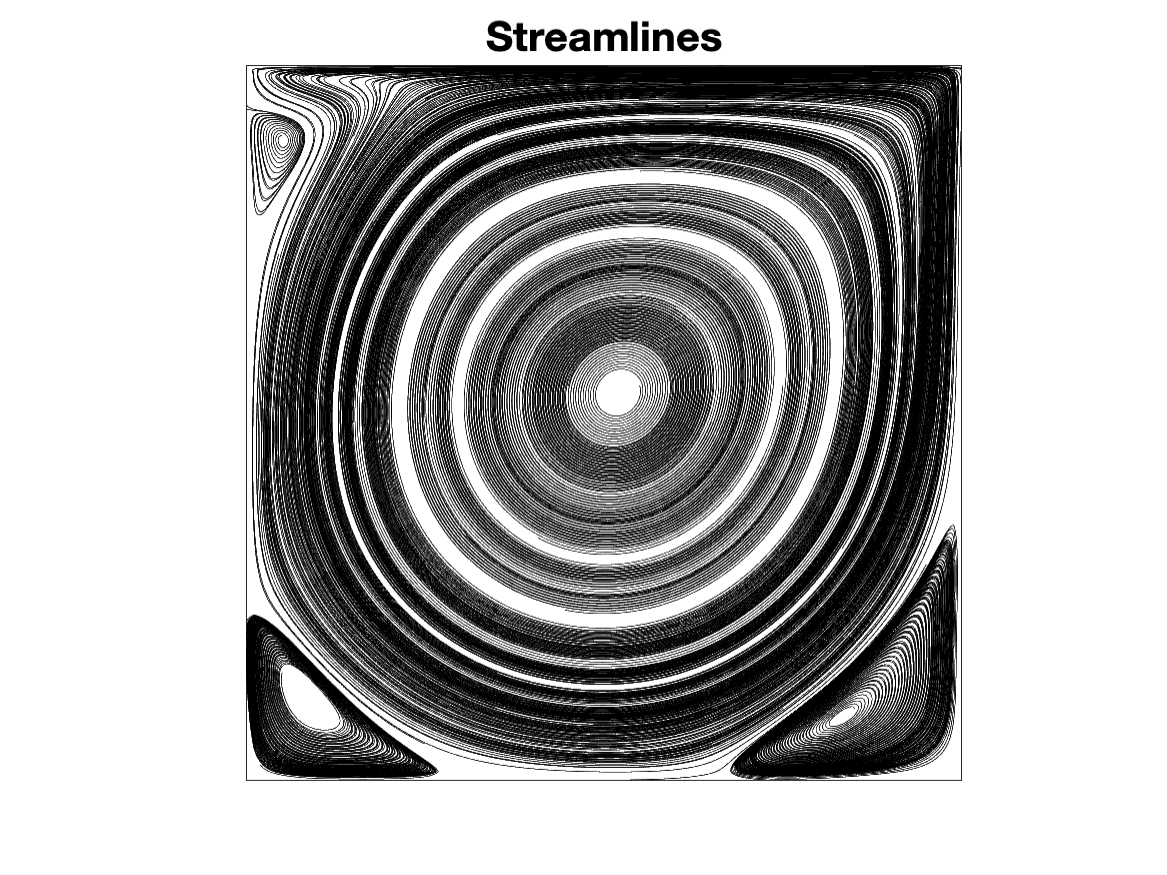}  
\includegraphics[width = .27\textwidth, height=.28\textwidth,viewport=115 45 465 390, clip]{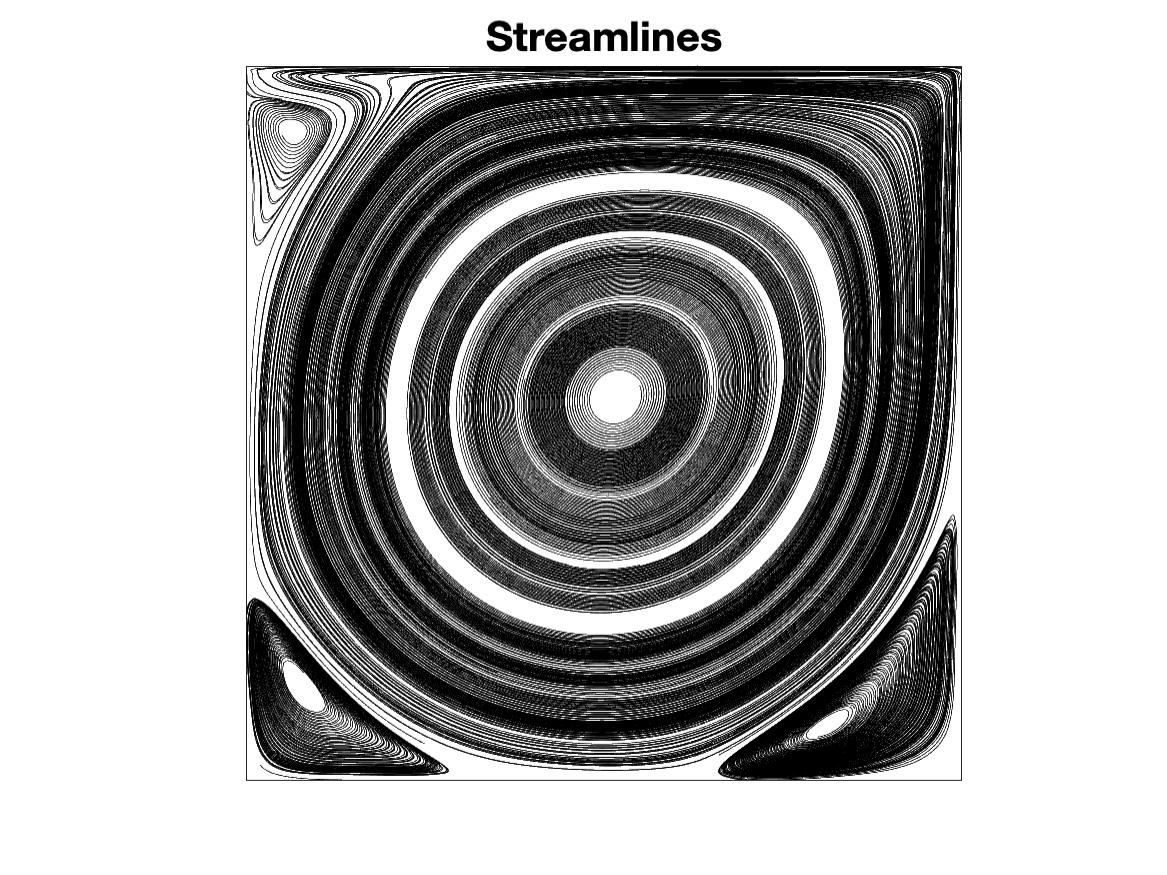}  
\includegraphics[width = .27\textwidth, height=.28\textwidth,viewport=115 45 465 390, clip]{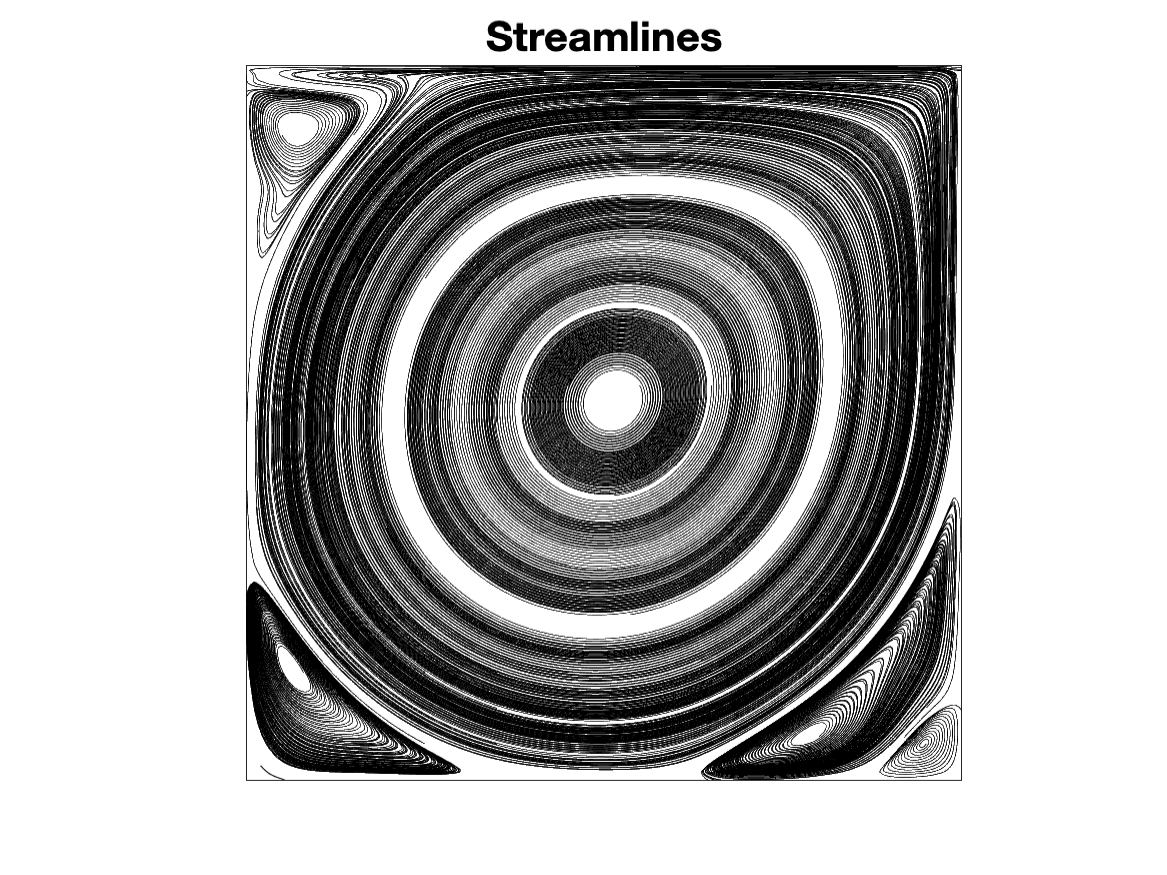}  
\caption{\label{nseplot1} The plot above shows streamlines of the solution of the 2D driven cavity problem at $Re=$3,000 (left), 5,000 (center), and 10,000 (right).}
\end{figure}

We first test convergence of solutions using CDA nudging $(0<\mu<\infty$) to the direct enforcement.  We do with this with two sets of parameters: $Re$=3,000 and $H=\frac{1}{10}$, and $Re$=5,000 and $H=\frac{1}{32}$.  Results are shown in figure \ref{convplotsmu} as $H^1$ error versus iteration number, and we observe that as $\mu$ increases, the convergence behavior becomes indistinguishable from that of direct enforcement (labeled $\mu=\infty$ in the plots).  This is consistent with our analysis which suggests that `raising $\mu$ large does no harm', and justifies the use of CDA through direct enforcement of the data measurements.

\begin{figure}[ht]
\center
\includegraphics[width = .48\textwidth, height=.35\textwidth,viewport=0 0 550 405, clip]{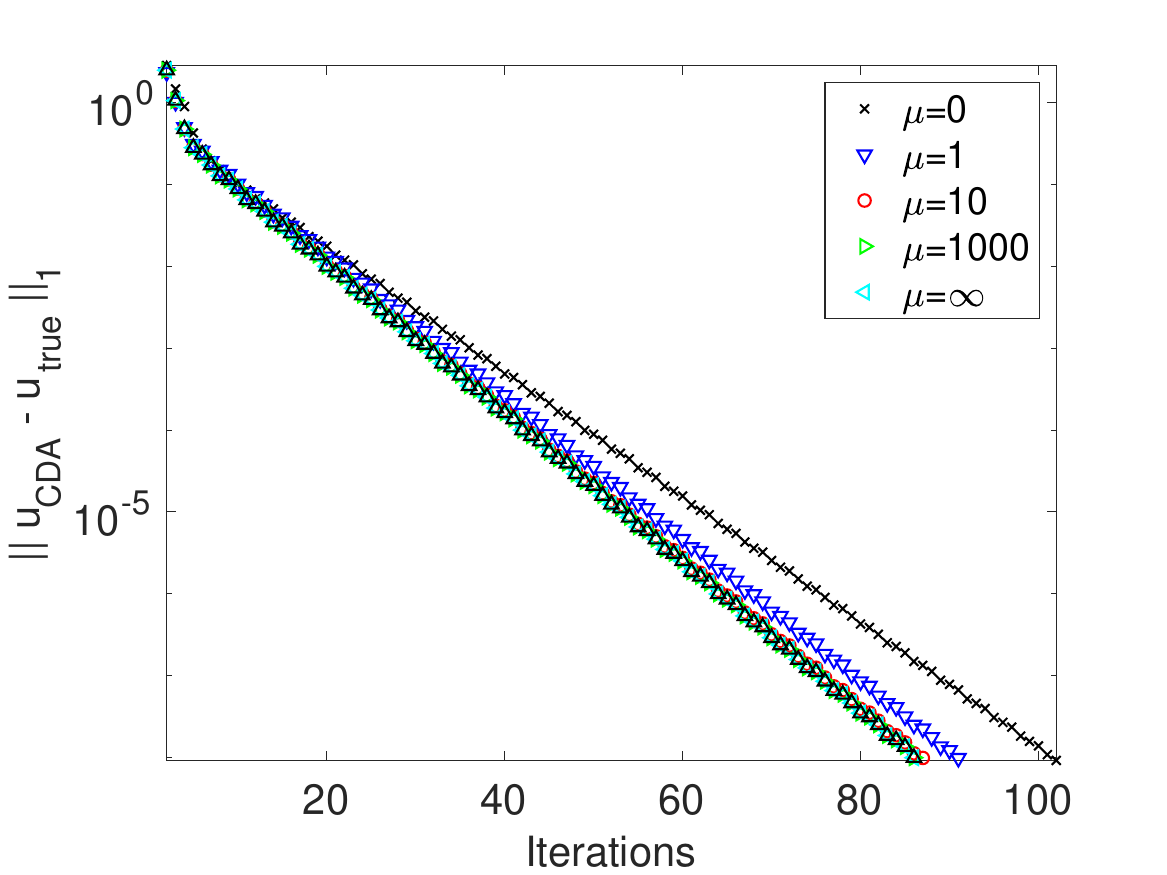}
\includegraphics[width = .48\textwidth, height=.35\textwidth,viewport=0 0 550 405, clip]{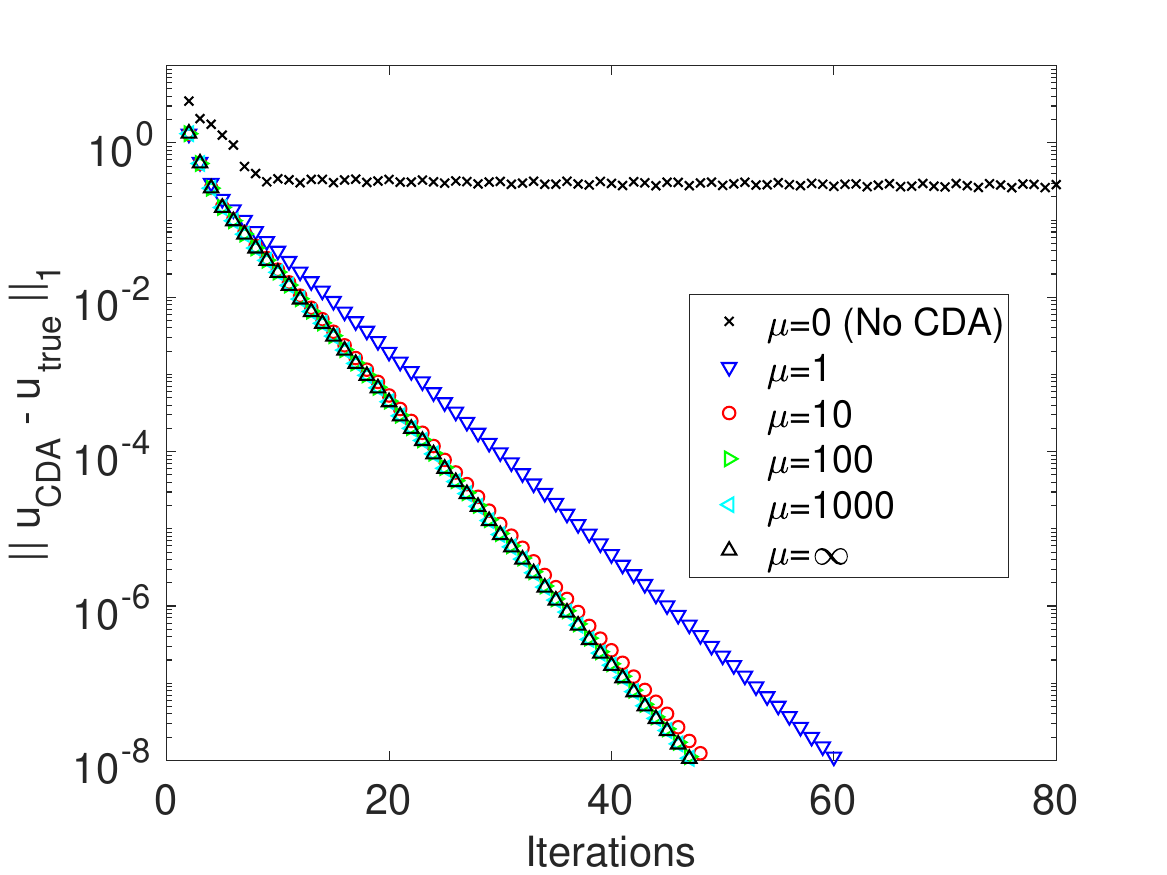}
\caption{\label{convplotsmu} Shown above is the convergence behavior for $Re=$3,000 and $H=\frac{1}{10}$ (left) and $Re=$5,000 and $H=\frac{1}{32}$ (right), with varying $\mu$ and with direct enforcement (labeled as $\mu=\infty$).}
\end{figure}

Next we test the CDA-Picard iteration convergence for $Re=$3,000 and $Re=$5,000 with direct enforcement of CDA and varying $H$.  Convergence results are shown in figure \ref{convplotsH}, and reveal a clear improvement from CDA.  As expected, the more data that is known, the faster the convergence.  For $Re$=3,000, significant improvement can be observed with $H=\frac{1}{16}$ and smaller, and for $Re=5,000$ the use of CDA enables convergence since without it the usual Picard method fails to converge.

\begin{figure}[ht]
\center
\includegraphics[width = .48\textwidth, height=.35\textwidth,viewport=0 0 550 405, clip]{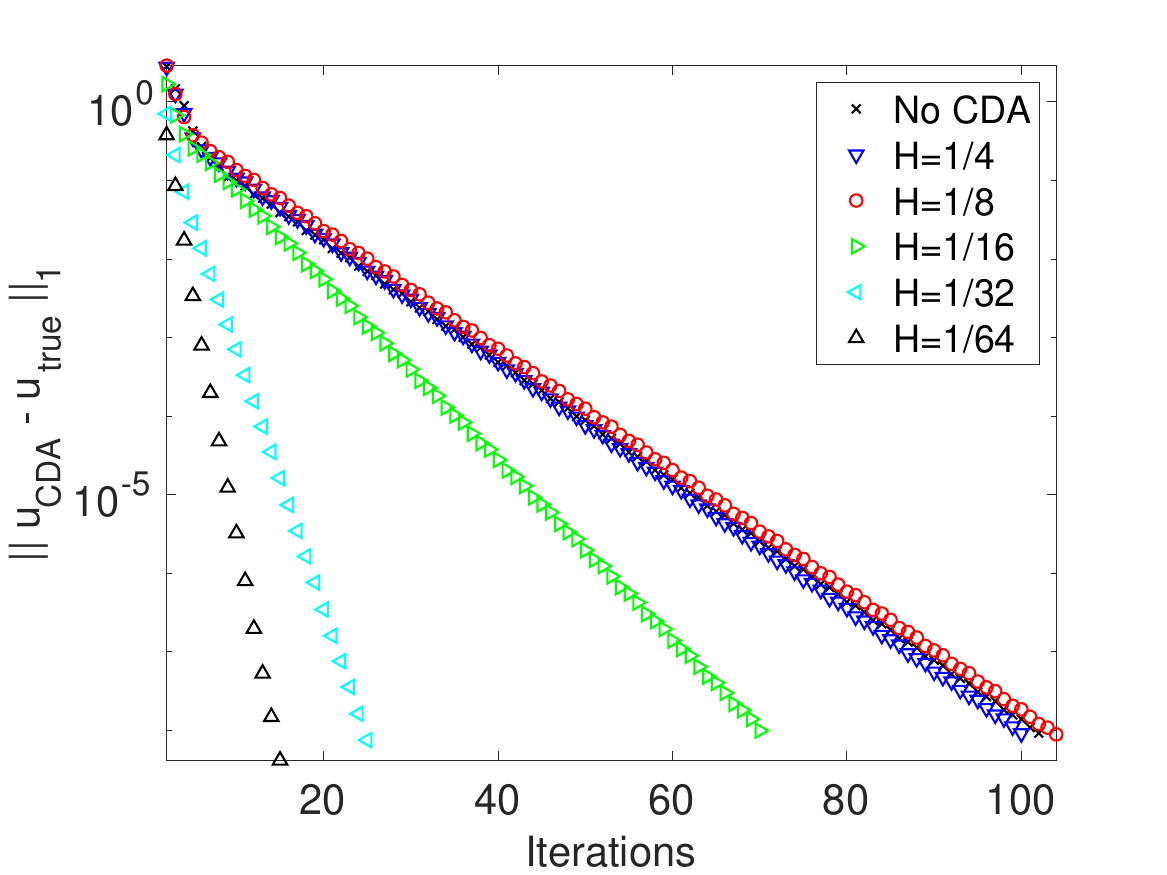}
\includegraphics[width = .48\textwidth, height=.35\textwidth,viewport=0 0 550 405, clip]{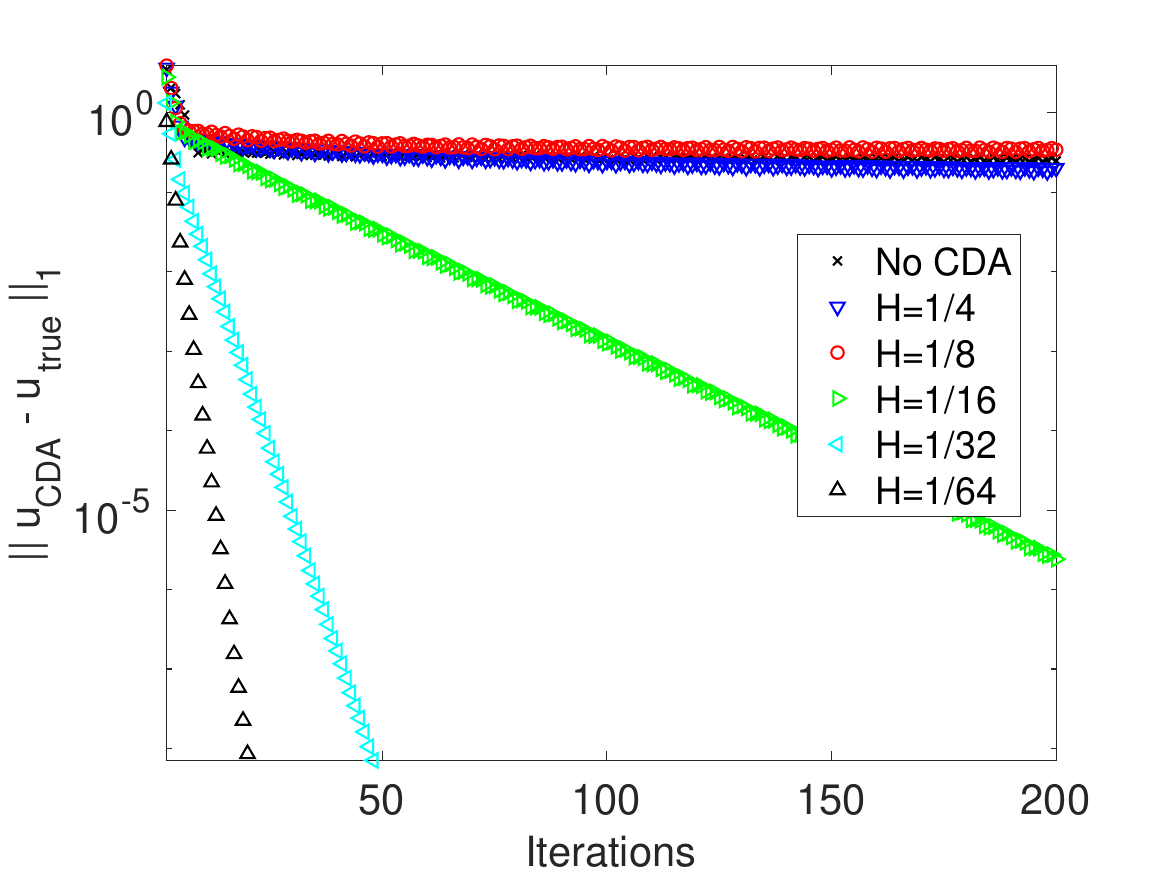}
\caption{\label{convplotsH} Shown above is the convergence behavior for $Re=3,000$ (left) and $Re=5,000$ (right), with varying $H$ and using direct enforcement for CDA implementation.}
\end{figure}

Lastly for the 2D driven cavity, we test with $Re=$10,000, which is well known to be a difficult problem for nonlinear solvers.   We first test the method with varying $H$ and direct CDA enforcement, and convergence results are shown in figure \ref{convplotsH2}.  We observe Picard and CDA-Picard do not perform well, and we need $H\le \frac{1}{32}$ to begin to see slow convergence but $H=\frac{1}{64}$ to see reasonably fast convergence.  Since $h$ is not much smaller than this $H$, this suggests CDA is not particularly useful in this case.  From \cite{PRX19,PR21}, we know that AA is known to help with this test problem, and so we use AA-Picard with depth $m=5$ and no relaxation and find much better convergence, see figure \ref{convplotsH2} at right.  We also test this method with CDA, i.e. CDA-AA-Picard, and see from the figure that CDA has a significant positive effect on convergence.  Hence, combining CDA with AA gives the best acceleration.

\begin{figure}[ht]
\center
\includegraphics[width = .48\textwidth, height=.35\textwidth,viewport=0 0 550 405, clip]{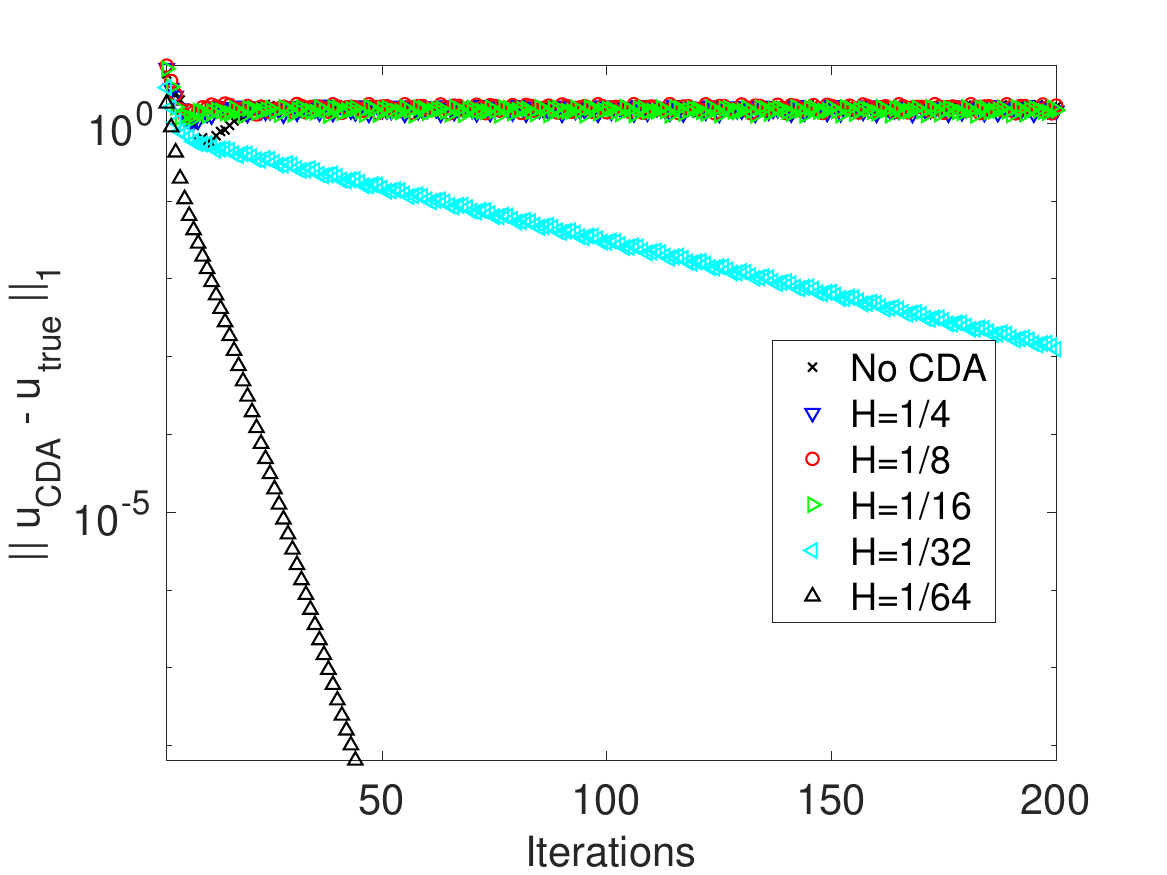}
\includegraphics[width = .48\textwidth, height=.35\textwidth,viewport=0 0 550 405, clip]{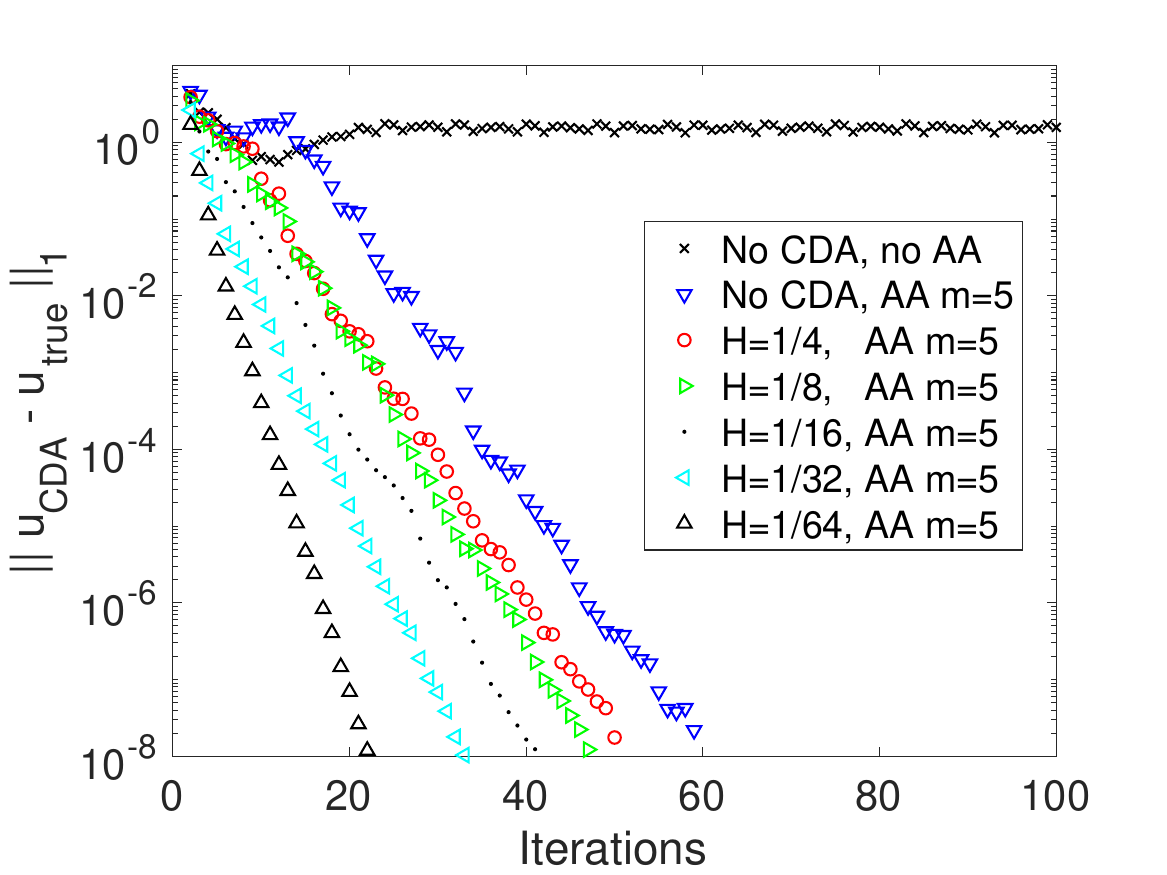}
\caption{\label{convplotsH2} Shown above is the convergence behavior for $Re=10,000$, with no Anderson acceleration (left) and with $m=5$ Anderson acceleration (right), with varying $H$ and using direct enforcement for CDA implementation.}
\end{figure}

\subsection{3D driven cavity}

		Our last test for CDA-Picard is using the 3D lid-driven cavity. In this problem, the domain is the unit cube, there is no forcing $ (f = 0)$, and homogeneous Dirichlet boundary conditions are enforced on all walls and $u=\langle 1,0,0\rangle$ on the moving lid.  We compute with $(P_3, P_2^{disc})$ Scott-Vogelius elements on barycenter refined (Alfeld split) tetrahedral meshes with 796,722  total dof that are weighted towards the boundary by using a Chebychev grid before tetrahedralizing. This velocity-pressure pair is known to be LBB stable from \cite{Z05_Alfed}.  We test CDA-Picard with varying $Re$= 400, 1000 and 1500, and solution plots we found are shown in Figure \ref{fig:midslideplanes} and are in agreement with the literature \cite{WongBaker2002}.  In these tests, we use direct enforcement of CDA.
		
\begin{figure}[H]
			\centering
			\includegraphics[width = 0.95\textwidth, height=.27\textwidth,viewport=100 0 1100 300, clip]{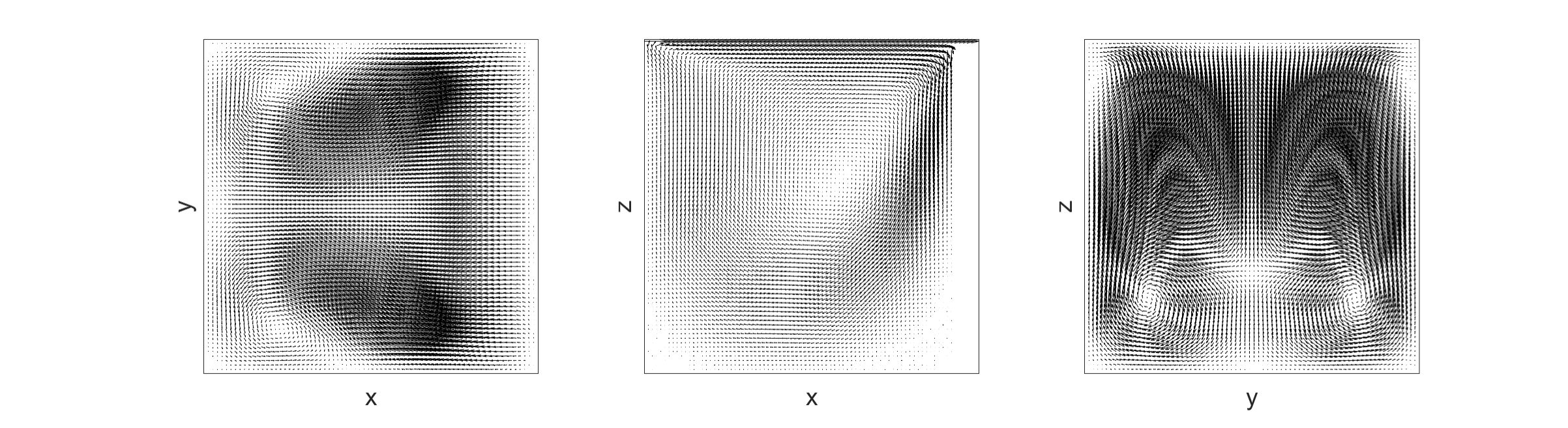}
			\includegraphics[width = 0.95\textwidth, height=.27\textwidth,viewport=100 0 1100 300, clip]{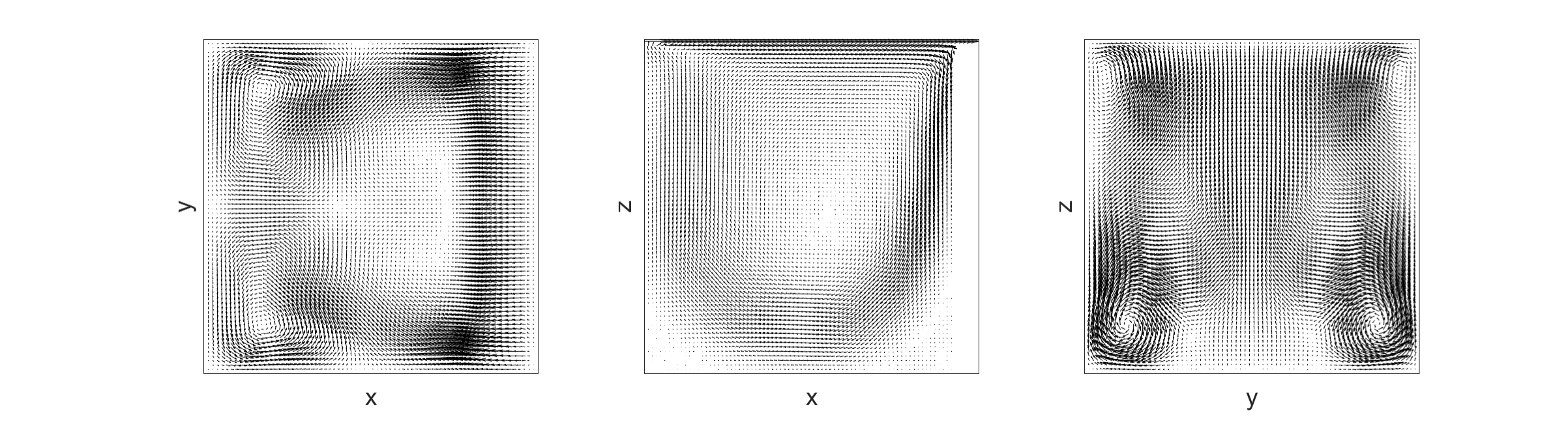}
			\includegraphics[width = 0.95\textwidth, height=.27\textwidth,viewport=100 0 1100 300, clip]{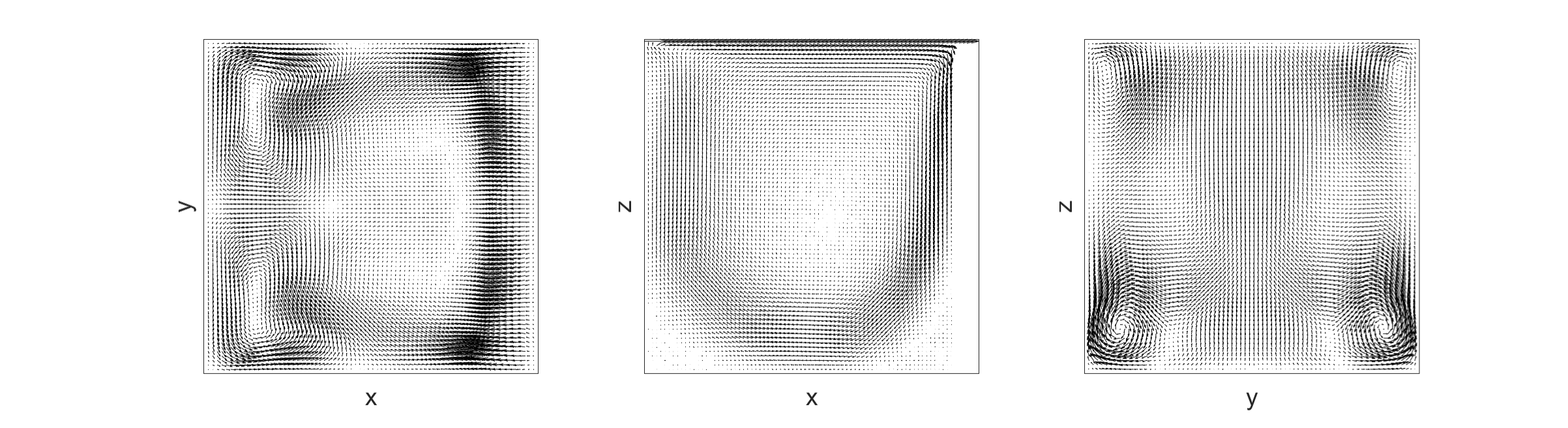}
			\caption{Shown above are the midsliceplane plots of solutions for the 3D driven cavity simulations at $Re$ = 400 (top), 1000 (middle) and 1500 (bottom).\label{fig:midslideplanes}}
\end{figure}

		Figure \ref{fig:re400CDA} shows the convergence for $Re=400$ with varying $H$.  We observe that without CDA, the Picard iteration fails, and also CDA-Picard fails with $H=1/4$.  Once $H=1/8$, convergence is achieved  and then improved for smaller $H$.
				
		\begin{figure}[H]
			\centering
			\includegraphics[width = .48\textwidth, height=.35\textwidth]{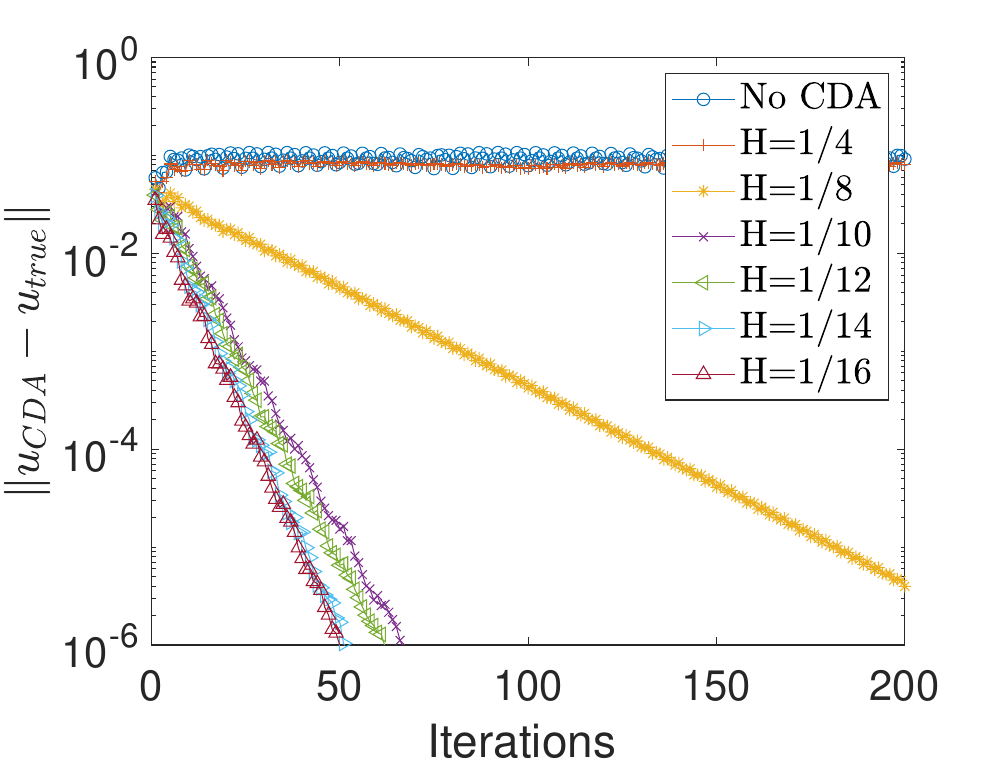}
			\caption{Convergence of CDA-Picard iterations for varying number of measurements for $Re=400$.\label{fig:re400CDA}}
		\end{figure}
	
		For $Re\ge 1000$, convergence with CDA-Picard fails without an excessive amount of data measurements (need $H\le \frac{1}{64}$, plots omitted).  Hence just as in the 2D case, we equip CDA-Picard with AA, using here $m=10$.  For $Re=1500$, we also use AA relaxation of $\beta_k=0.5$.  In figure \ref{fig:re1000CDAm1beta1}, we show the convergence of CDA-Picard for $Re=1000$ and 1500 with varying $H$.  Here we observe that AA-Picard converges in both cases, and is accelerated by using CDA-AA-Picard.

		\begin{figure}[H]
			\centering
			\includegraphics[width = .48\textwidth, height=.35\textwidth]{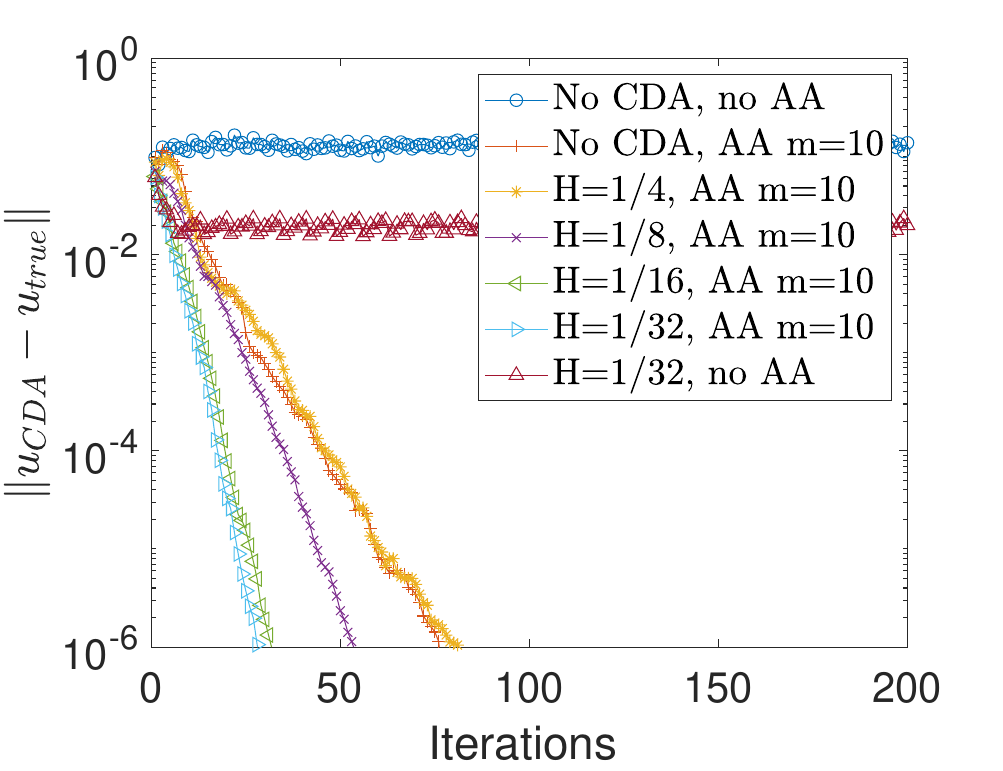}
						\includegraphics[width = .48\textwidth, height=.35\textwidth]{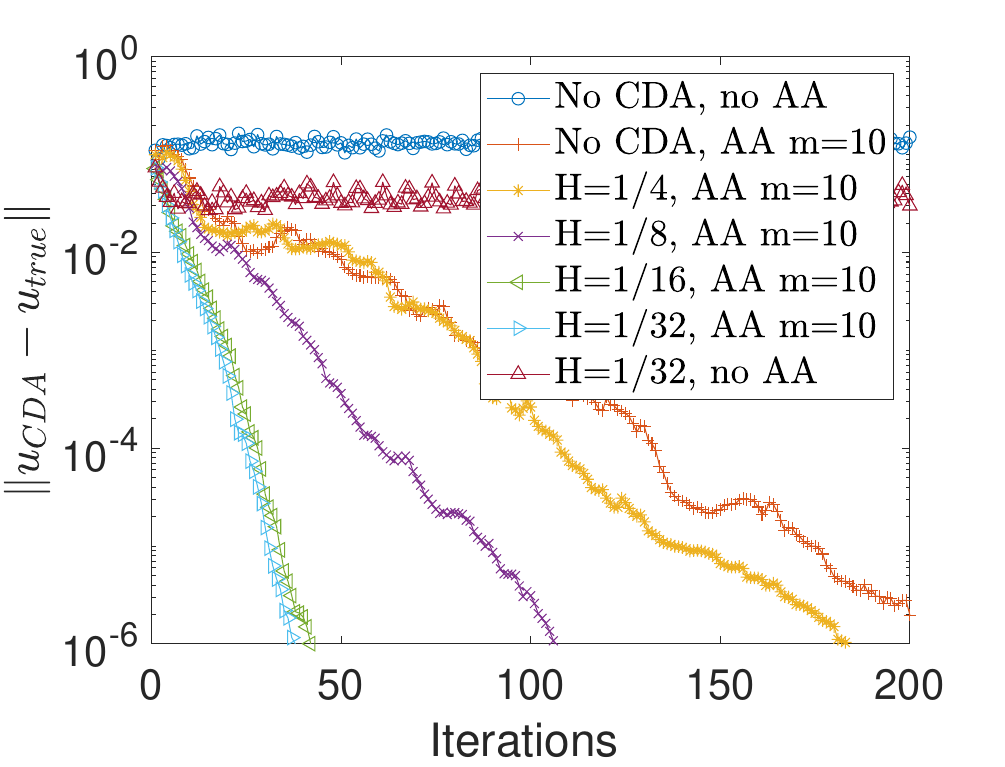}
			\caption{Convergence of CDA-Picard with AA ($m=10$)  for $Re=1000$ (left) and $Re=1500$ (right) with varying $H$.\label{fig:re1000CDAm1beta1}}
		\end{figure}

\section{Analysis of CDA-Newton}\label{CDA-Newton}

Although Newton's method converges quadratically in an asymptotic sense, it has a well-known issues (in general, but in particular for steady NSE) in a good initial guess is required for convergence and may also have stability issues, see Lemmas \ref{Nstable} and \ref{nc}.  This is opposed to Picard, which is stable for steady NSE.  For Newton, these lemmas show a sufficient condition for convergence of Newton for steady NSE is that $\frac{M}{\nu(1-\alpha)}\|\nabla (u - u_{0}\|<1$ and $\alpha \leq \frac18$.  While these conditions are not sharp, they are believed reasonably sharp and usually accurate to within an order of magnitude.  In this section we will show that by incorporating CDA into the iteration, we can reduce the closeness condition required for the initial guess and remove the smallness restriction on $\alpha$ needed for quadratic convergence.  We will also show results of numerical tests that reveal significant improvements that CDA provides.  

CDA-Newton is given as follows.  Let $u$ be a steady NSE solution, and suppose we are given solution measurement data $I_H u$ and $u_k\in V$.  Then we want to find $u_{k+1}\in V$ satisfying for all $v\in V$,
\begin{equation}
\begin{split}\label{CDAN}
	a\left({u}_{k+1},v\right)+ b\left({u}_{k},{u}_{k+1},v\right)&+b\left({u}_{k+1},{u}_{k},v\right)\\&+\mu(I_Hu_{k+1}-I_Hu,I_Hv)=\left\langle{f},v\right\rangle+b\left({u}_{k},{u}_{k},v\right).  
\end{split}
\end{equation}

Our first result shows that under the same smallness condition of Newton (for a fair comparison), CDA-Newton has a larger domain of convergence for the initial iterate.

\vspace{2mm}
\begin{theorem}\label{sn}
Let $u$ be a steady NSE solution and suppose $I_H(u)$ is known and $\mu$ is chosen sufficiently large so that $\mu\ge \frac{\beta}{4C_I^2H^2}$, where $\beta =\nu (1-2\alpha)$.  Assume that $8\alpha<1$,  and 
\begin{align}\label{t1}
	\frac{M\sqrt{2C_IH}}{\beta}\|\nabla (u - u_0) \|<1.
\end{align} 
Then CDA-Newton (\ref{CDAN}) is stable and converges to a solution $u$ of NSE (\ref{wd}) quadratically:
\begin{align}\label{cs}
	\|\nabla (u - u_{k+1})\|\leq \frac{M\sqrt{2C_IH}}{\beta} \|\nabla (u - u_k)\|^2. 
\end{align}
\end{theorem}
\vspace{2mm}
\begin{remark}
The theorem above shows that decreasing $H$ enlarges the domain of convergence in the initial condition, and also decreases the constant in the quadratic convergence.  Additionally, the choice of $\mu=\infty$ and direct enforcement of the CDA can be done for CDA-Newton just like in CDA-Picard, and not affect the result of the theorem.
\end{remark}

\begin{proof}
The stability proof for Newton in Lemma \ref{Nstable} can be repeated, thanks to the nudging term having $I_H$ on the test function (it immediately produces a left hand side non-negative term that can be dropped), to get $\| \nabla u_k \|\le 2\nu^{-1} \| f \|_{-1}$ for any $k$. 

Set $e_k = u - u_k$.  Subtracting equation (\ref{CDAN}) from (\ref{wd}) and rearranging terms similar to \eqref{N2} gives
\begin{equation}\label{PNewton}
	\begin{split}
		0&=a(e_{k+1},v)+\mu (I_He_{k+1},I_Hv)+b(e_k,e_k,v)+b(e_{k+1},u_k,v)+b(u_k,e_{k+1},v).
	\end{split}
\end{equation}
Taking $v=e_{k+1}$ in equation (\ref{PNewton}), applying a standard bound on the trilinear term $b$, Young's inequality, and (\ref{snewton}), we obtain the bound 
\begin{equation}
	\begin{split}
		\nu\|\nabla e_{k+1}\|^2+\mu\|I_He_{k+1}\|^2&=a(e_{k+1},v)+\mu (I_He_{k+1},I_Hv)\\&=-b(e_k,e_k,e_{k+1})-b(u_k,e_{k+1},e_{k+1})-b(e_{k+1},u_k,e_{k+1})\\
		&=-b(e_k,e_k,e_{k+1})-b(e_{k+1},u_k,e_{k+1})\\
		&\leq M\|\nabla e_{k}\|^2\|\nabla e_{k+1}\|^{\frac{1}{2}}\| e_{k+1}\|^{\frac{1}{2}}+M\|\nabla u_k\|\|\nabla e_{k+1}\|^2\\
		&\leq A\|\nabla e_{k}\|^4+\frac{M^2}{4A}\|\nabla e_{k+1}\|\| e_{k+1}\|+\frac{2M}{\nu}\|f\|_{-1}\|\nabla e_{k+1}\|^2,
	\end{split}
\end{equation}
which gives us 
\begin{equation}\label{1}
	\begin{split}
		\left(\nu-\frac{2M}{\nu}\|f\|_{-1}\right)\|\nabla e_{k+1}\|^2+\mu\|I_He_{k+1}\|^2\leq A\|\nabla e_{k}\|^4+\frac{M^2}{4A}\|\nabla e_{k+1}\|\| e_{k+1}\|,
	\end{split}
\end{equation}
where $A>0.$  
We lower bound the left side of \eqref{1} analogously to what is done above for CDA-Picard in \eqref{lowerb} to get
%Note that $\beta=\nu-\frac{2M}{\nu}\|f\|_{-1}=\nu(1-2\alpha)>0$ since $8\alpha<1$. Applying inequality (\ref{interpolationi}) and the triangle inequality, we derive a lower bound for $\beta\|\nabla e_{k+1}\|^2+\mu\|I_He_{k+1}\|^2$: 
\begin{equation}\label{22}
	\begin{split}
\frac{3\beta}{4}\|\nabla e_{k+1}\|^2+\frac\lambda2\|e_{k+1}\|^2
\le 
		\beta\|\nabla e_{k+1}\|^2+\mu\|I_He_{k+1}\|^2,
	\end{split}
\end{equation}
where $\lambda=\min\{\mu,
\frac{\beta}{4C_I^2H^2}\}$.  Using Young's inequality on $\frac{M^2}{4A}\|\nabla e_{k+1}\|\| e_{k+1}\|$ results in
\begin{equation}\label{33}
	\begin{split}
		\frac{M^2}{4A}\|\nabla e_{k+1}\|\| e_{k+1}\|&\leq \frac{\beta}{4}\|\nabla e_{k+1}\|^2+\frac{M^4}{16A^2\beta}\| e_{k+1}\|^2.
	\end{split}
\end{equation}
Combining (\ref{1}), (\ref{22}), and (\ref{33}) leads to 
\begin{equation}\label{44}
	\begin{split}
		\frac{\beta}{2}\|\nabla e_{k+1}\|^2+\left(\frac\lambda2-\frac{M^4}{16A^2\beta}\right)\|e_{k+1}\|^2\leq A\|\nabla e_{k}\|^4.
	\end{split}
\end{equation}
Since we assume $\mu$ is chosen large enough, we can take $\lambda=\frac{\beta}{4C_I^2H^2}$.  Choosing $A=\frac{C_IM^2H}{\sqrt{2}\beta}$ now provides
%If 
%\begin{align}
%	\lambda-\frac{M^4}{16A^2\beta}\geq0, ~i.e,~ H<\frac{2A\beta}{C_IM^2},
%\end{align}
%then
\begin{align}
	\|\nabla e_{k+1}\|\leq \frac{M\sqrt{2C_IH}}{\beta} \|\nabla e_k\|^2. 
\end{align}
This completes the proof since it infers both the constant for quadratic convergence as well as the closeness of the initial condition that yields a contraction condition.
\end{proof}

Theorem \ref{sn} requires a smallness condition $8\alpha< 1$ for the proof. In our next result, we show that with sufficient observation data, CDA-Newton can remove the restriction on $\alpha$.
\begin{theorem}\label{snn}
Let $u$ be a steady NSE solution and suppose $I_H(u)$ is known and $H$ is sufficiently small (or the initial guess is sufficiently good) so that $A$ can be chosen to satisfy
\begin{equation}
~\label{3di} 0<A < \frac{\nu}{8 \|\nabla (u - u_0)\|^2}
\end{equation}
but with the inequality
\[
\frac{\nu}{16C_I^2H^2}-\frac{M^4}{16A^2\nu}-\frac{27M^4}{4\nu^3}\|\nabla (u - u_0)\|^4-\frac{27}{4}\nu\alpha^4>0
\]
holding.  Suppose further that $\mu\ge \frac{\nu}{8C_I^2H^2}$.  Then CDA-Newton (\ref{CDAN}) converges to $u$ quadratically: 
	\begin{align}\label{css}
		\| \nabla (u - u_{k+1})\|\leq \sqrt{\frac{8A}{\nu}}\|\nabla (u - u_{k})\|^2.
	\end{align}
\end{theorem}
\vspace{1mm}
%\begin{remark}
% Theorem \ref{snn}  shows that the CDA-Newton can fix the stability issue that usual Newton iteration may encounter, provided sufficiently small $H$.
% \end{remark}
\begin{proof}
Set $e_k=u-u_k$.	Subtracting equation (\ref{CDAN}) from (\ref{wd}) (same as  (\ref{PNewton})) gives 
\begin{equation}\label{SNewton}
	\begin{split}
		0&=a(e_{k+1},v)+\mu (I_He_{k+1},I_Hv)+b(e_k,e_k,v)+b(e_{k+1},u_k,v).
	\end{split}
\end{equation}
Note we do not assume upper bound on $u_k$ here, and this varies the proof from the previous theorem's proof. By adding and subtracting term $b(e_{k+1},u,e_{k+1})$ in (\ref{SNewton}) we have
	\begin{equation}\label{3dold}
		\begin{split}
		0&=a(e_{k+1},v)+\mu (I_He_{k+1},I_Hv)+b(e_k,e_k,v)-b(e_{k+1},e_k,v)+b(e_{k+1},u,v).
		\end{split}
	\end{equation}
%Setting $v=e_{k+1}$	on (\ref{3dold}), rearranging terms, and  applying a standard bound on the trilinear term $b$, Young's inequality, and inequality (\ref{Pri}), we obtain 
%	\begin{equation}\label{2c}
%		\begin{split}
%			&	\nu\|\nabla e_{k+1}\|^2+\mu\|I_He_{k+1}\|^2
%			%&=-b(e_k,e_k,e_{k+1})+b(e_{k+1},e_k,e_{k+1})-b(e_{k+1},u,e_{k+1})\\
%			\leq M\|\nabla e_{k}\|^2\|\nabla e_{k+1}\|^{\frac{1}{2}}\| e_{k+1}\|^{\frac{1}{2}}\\&~~~~+M\|\nabla e_k\|\|\nabla e_{k+1}\|^{\frac{3}{2}}\| e_{k+1}\|^{\frac{1}{2}}+M\|\nabla u\|\|\nabla e_{k+1}\|^{\frac{3}{2}}\| e_{k+1}\|^{\frac{1}{2}}\\
%			&\leq A\|\nabla e_{k}\|^4+\frac{M^2}{4A}\|\nabla e_{k+1}\|\| e_{k+1}\|+\frac{\nu}{4}\|\nabla e_{k+1}\|^2+\frac{27}{4\nu^3}M^4\|\nabla e_k\|^4\|e_{k+1}\|^2\\&~~~~+\frac{\nu}{4}\|\nabla e_{k+1}\|^2+\frac{27}{4\nu^3}M^4\|\nabla u\|^4\|e_{k+1}\|^2\\
%		%	&\leq A\|\nabla e_{k}\|^4+\frac{3\nu}{4}\|\nabla e_{k+1}\|^2+\frac{M^4}{16A^2\nu}\| e_{k+1}\|^2+\left(\frac{4}{\nu}\right)^3M^4\|\nabla e\|^4\|e_{k+1}\|^2\\&~~~~+M^4\nu^{-4}\|f\|_{-1}^4\left(\frac{4}{\nu}\right)^3\|e_{k+1}\|^2\\
%			&\leq\frac{3\nu}{4}\|\nabla e_{k+1}\|^2+A\|\nabla e_{k}\|^4+\frac{M^4}{16A^2\nu}\| e_{k+1}\|^2+\frac{27M^4}{4\nu^3}\|\nabla e_k\|^4\|e_{k+1}\|^2+\frac{27}{4}\nu\alpha^4\|e_{k+1}\|^2,
%		\end{split}
%	\end{equation}
%where $A>0$.
Setting $v=e_{k+1}$	on (\ref{3dold}), rearranging terms, and  applying a standard bound on the trilinear term $b$, Young's inequality, and inequality (\ref{Pri}), we obtain 
{\color{red}	{	\begin{equation}\label{2c}
		\begin{split}
			&	\nu\|\nabla e_{k+1}\|^2+\mu\|I_He_{k+1}\|^2
			%&=-b(e_k,e_k,e_{k+1})+b(e_{k+1},e_k,e_{k+1})-b(e_{k+1},u,e_{k+1})\\
			\leq M\|\nabla e_{k}\|^2\|\nabla e_{k+1}\|^{\frac{1}{2}}\| e_{k+1}\|^{\frac{1}{2}}\\&~~~~+M\|\nabla e_k\|\|\nabla e_{k+1}\|^{\frac{3}{2}}\| e_{k+1}\|^{\frac{1}{2}}+M\|\nabla u\|\|\nabla e_{k+1}\|^{\frac{3}{2}}\| e_{k+1}\|^{\frac{1}{2}}\\
			&\leq A\|\nabla e_{k}\|^4+\frac{M^2}{4A}\|\nabla e_{k+1}\|\| e_{k+1}\|+\frac{\nu}{4}\|\nabla e_{k+1}\|^2+\frac{27}{4\nu^3}M^4\|\nabla e_k\|^4\|e_{k+1}\|^2\\&~~~~+\frac{\nu}{4}\|\nabla e_{k+1}\|^2+\frac{27}{4\nu^3}M^4\|\nabla u\|^4\|e_{k+1}\|^2\\
		%	&\leq A\|\nabla e_{k}\|^4+\frac{3\nu}{4}\|\nabla e_{k+1}\|^2+\frac{M^4}{16A^2\nu}\| e_{k+1}\|^2+\left(\frac{4}{\nu}\right)^3M^4\|\nabla e\|^4\|e_{k+1}\|^2\\&~~~~+M^4\nu^{-4}\|f\|_{-1}^4\left(\frac{4}{\nu}\right)^3\|e_{k+1}\|^2\\
			&\leq\frac{3\nu}{4}\|\nabla e_{k+1}\|^2+A\|\nabla e_{k}\|^4+\frac{M^4}{16A^2\nu}\| e_{k+1}\|^2+\frac{27M^4}{4\nu^3}\|\nabla e_k\|^4\|e_{k+1}\|^2+\frac{27}{4}\nu\alpha^4\|e_{k+1}\|^2,
		\end{split}
	\end{equation}}}
where $A>0$.
Analogous to what is done for CDA-Picard in section 3 with sufficiently large $\mu$, we lower bound $\nu\|\nabla e_{k+1}\|^2+\mu\|I_He_{k+1}\|^2$ as:
	\begin{align}\label{sn1}
		\nu\|\nabla e_{k+1}\|^2+\mu\|I_He_{k+1}\|^2\geq \frac{	\nu}{16C_I^2H^2}\|e_{k+1}\|^2+\frac{7\nu}{8}\|\nabla e_{k+1}\|^2.
	\end{align}
	Combining  (\ref{2c}) and (\ref{sn1}) gives us 
	\begin{equation}\label{11}
		\begin{split}
			&\left(\frac{\nu}{16C_I^2H^2}-\frac{M^4}{16A^2\nu}-\frac{27M^4}{4\nu^3}\|\nabla e_k\|^4-\frac{27}{4}\nu\alpha^4\right)\|e_{k+1}\|^2+\frac{	\nu}{8}\|\nabla e_{k+1}\|^2\leq A\|\nabla e_{k}\|^4.
		\end{split}
	\end{equation}
	Next, choose $A$ and $H$ such that the following inequalities are satisfied (it is assumed $H$ is sufficiently small so this is possible):
	\begin{align}
		&\label{10}	\sqrt{\frac{8A}{\nu}}\|\nabla e_0\|<1 ~~\text{and~ ~ }\frac{\nu}{16C_I^2H^2}-\frac{M^4}{16A^2\nu}-\frac{27M^4}{4\nu^3}\|\nabla e_0\|^4-\frac{27}{4}\nu\alpha^4\geq 0.
	\end{align}
	By induction, once both inequalities in (\ref{10}) hold, we have, for any $k$ that
	\begin{align}
		\frac{\nu}{16C_I^2H^2}-\frac{M^4}{16A^2\nu}-\frac{27M^4}{4\nu^3}\|\nabla e_0\|^4-\frac{27}{4}\nu\alpha^4\geq 0,
	\end{align}
	and thus
	\begin{align}\label{3dl}
		\|\nabla e_{k+1}\|\leq \sqrt{\frac{8A}{\nu}}\|\nabla e_k\|^2.  
	\end{align}
	This completes the proof.
\end{proof}

\subsection{CDA-Newton numerical tests}\label{TCDA_N}

We now test the CDA-Newton method on the same 2D driven cavity test problem from section 3.  Here we use Re=1000, 3000 and 5000, and $\frac{1}{64}$ uniform triangulation with $(P_2,P_1)$ Taylor-Hood elements.  In our tests, Newton (without CDA) works well up until Re=700, after which it fails (tests omitted).  We do not include any line search in our implementation; we implement the algorithm as stated in section 2.

Figure \ref{NC2} shows convergence results for the cases of Re=1000, 3000 and 5000, with varying $H$.  In each case, we observe that usual Newton fails, and with enough data, CDA-Newton converges.  For {\color{red} Re=1000 and 3000, $H\le \frac18$ (49 measurement locations) is needed and $H\le\frac{1}{16}$ (225 measurement locations) is needed for 5000}.  This is consistent with our analysis that as $\nu$ decreases, $H$ must also decrease to achieve convergence.  We also observe that when CDA-Newton does converge, it converges quadratically as expected.

\begin{figure}[H]
\includegraphics[height=4cm,viewport=0 0 400 290, clip]{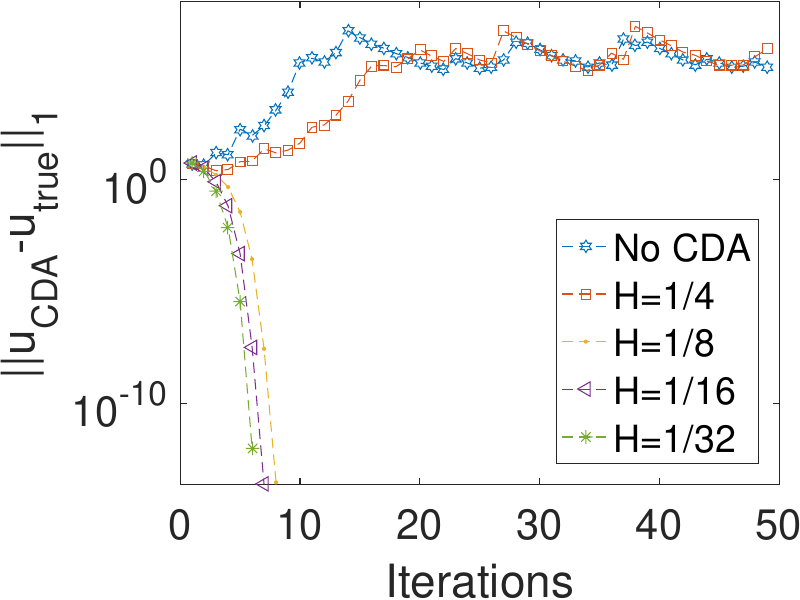}
\includegraphics[height=4cm,viewport=0 0 400 290, clip]{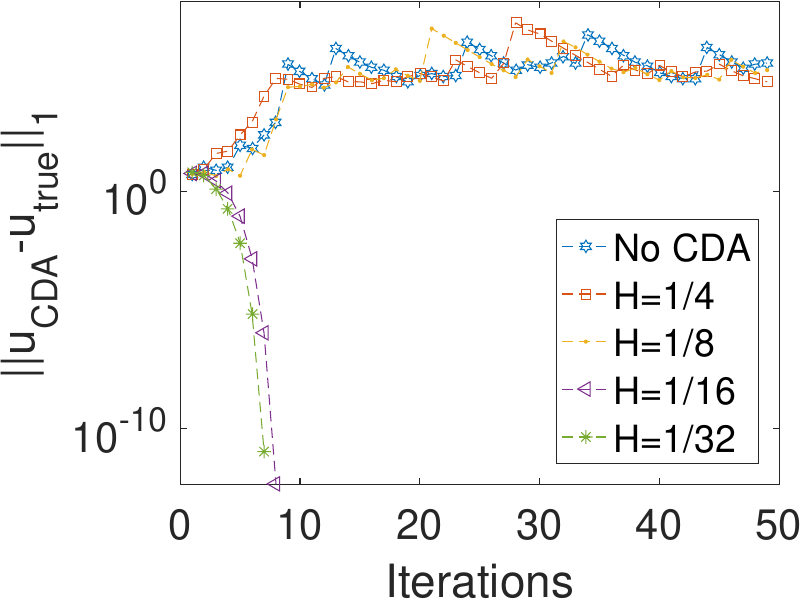}
\includegraphics[height=4cm,viewport=0 0 400 290, clip]{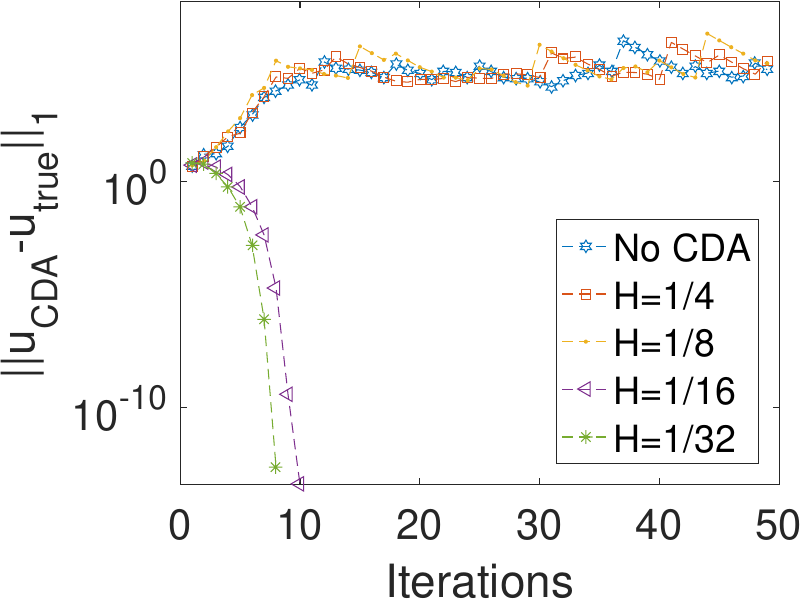}
\caption{Shown above are CDA-Newton convergence plots at varying Re=1000 (left), 3000 (center) and 5000 (right).}\label{NC2}
\end{figure}

\section{Conclusions}\label{conclusion}
We proposed, analyzed and tested CDA-Picard and CDA-Newton methods to incorporate measurement data into nonlinear solvers.  We found that the use of {\color{red} CDA provided accelerated convergence and even enabled convergence in some cases.  In no tests did CDA negatively affect convergence.  For CDA-Picard, we proved that the linear convergence rate was improved by a scaling factor of (at least) $H^{1/2}$, and that convergence can be achieved for larger $\alpha$ when CDA is used.  In practice, this means that if usual Picard already converges, using CDA with any amount of data will accelerate convergence, and the more data is included the faster the convergence will be.  If usual Picard does not converge, then the analysis proves that if enough measurement data is obtained so that $H\alpha^2 \le O(1)$, then CDA-Picard will converge.  However, the numerical tests suggest this sufficient condition on $H$ is a very pessimistic bound.  Additionally, we note that combingation of Anderson acceleration together with CDA-Picard provided the best results for improving Picard.} 

For CDA-Newton, we proved that the condition of usual Newton that the initial iterate be `close enough' can be significantly weakened, as reducing $H$ increases the size of the domain where initial guesses can be and the iteration will still converge.  We also proved for CDA-Newton that with enough data measurements, the smallness assumption can be removed.  Multiple numerical tests illustrated these results, {\color{red} and again the analysis bounds for $H$ and $\alpha$ were much more pessimistic than what is observed in numerical tests.}

Additionally, motivated by our analysis, the implementation of CDA was {\color{red} able to (implicitly) use $\mu=\infty$ by doing a direct implementation of the measurements into the linear systems, just as a Dirichlet boundary condition would be implemented.}  We observed no adverse effects from this implementation.

For future work, we will apply our analysis techniques of this paper, together with $\mu=\infty$, to CDA applied to fully discretized time dependent PDEs and other steady PDEs.  This work is underway, and already showing very promising results.  A second future direction is exploring the relationship between the CDA-Picard results and nonsingular steady NSE solutions.  That is, we have proven that if $\sqrt{ {\color{red}\sqrt{2}}C_I} H^{1/2}\alpha<1$, then CDA-Picard is globally convergent to the steady NSE solution that also agrees with the measurement data (which is known to be unique \cite{L23}).  Hence if $H$ is small enough, then there can be only 1 steady NSE solution that satisfies the finite number of measurement values, even when the steady NSE has multiple solutions.  {\color{red} Lastly, we plan to look deeper into a more precise condition on $H$ that will provide convergence of CDA-Picard when Picard fails.  As mentioned above, there is a significant gap between the sufficient conditions for convergence of CDA-Picard provided by the theory and what the numerical tests show is necessary.}

\section{Acknowledgements}
All authors were partially supported by NSF grant DMS 2152623.\\ \ \\
The authors thank Professors Julia Novo and Bosco Garc\'ia-Archilla for helpful discussions regarding this work.

%\bibliographystyle{plain}
%\bibliography{graddiv,Xuejian_Li_ref}

\section{Appendix}

	\begin{proof}[Proof of Lemma \ref{PLemma}]
		We first show the unconditional stability. Choosing $v=u_{k+1}$ in (\ref{PicardO}) and then rearranging terms leads to 
		\begin{equation}\label{sbb3}
			\begin{split}
				\nu\|\nabla {u}_{k+1} \|^2&=a\left({u}_{k+1},{u}_{k+1}\right)=\left\langle{f},{u}_{k+1}\right\rangle-b\left({u}_{k},{u}_{k+1},{u}_{k+1}\right)
				\leq \|f\|_{-1}\|\nabla u_{k+1}\|,
			\end{split}
		\end{equation}
		which gives the bound $\|\nabla {u}_{k+1} \|\leq \nu^{-1}\|f\|_{-1}$ for any $k$.
		
		Next, we prove convergence. Subtracting (\ref{wd}) from (\ref{PicardO}) results in 
		\begin{equation}\label{sbb}
			\begin{split}
				0&=a(u,v)-a(u_{k+1},v)+b(u,u,v)-b(u_k,u_{k+1},v)\\
				%	&=a(u-u_{k+1},v)+b(u,u,v)-b(u_k,u,v)+b(u_k,u,v)-b(u_k,u_{k+1},v)\\
				&=a(u-u_{k+1},v)+b(u-u_k,u,v)+b(u_k,u-u_{k+1},v)\\
				&=a(e_{k+1},v)+b(e_k,u,v)+b(u_k,e_{k+1},v),
			\end{split}
		\end{equation}
		where $e_k:=u - u_k$.
		Taking $v=e_{k+1}$ in equation (\ref{sbb}) and applying a bound on the  trilinear term,  we have 
		\begin{equation}\label{upperr}
			\begin{split}
				&\nu\|\nabla e_{k+1}\|^2=a(e_{k+1},e_{k+1})=-b(e_k,u,e_{k+1})\\
				&\leq M\|\nabla e_k\|\|\nabla u\| \|\nabla e_{k+1}\|\leq M\nu^{-1}\|f\|_{-1}\|\nabla e_k\| \|\nabla e_{k+1}\|,
			\end{split}
		\end{equation}
		which results in the inequality 
		$
		\|\nabla e_{k+1}\|\leq M\nu^{-2}\|f\|_{-1}\|\nabla e_k\|=\alpha \|\nabla e_k\|.$
		This completes the proof. 
	\end{proof}

	We now prove Lemma \ref{nc}.  This proof consists of two parts: stability and convergence, and we begin with stability.
	
	\begin{lemma}\label{Nstable}
		If $8\alpha\leq 1$, the Newton method (\ref{NewtonO}) is stable, i.e., 
		\begin{align}\label{snewton}
			\|\nabla u_k\|\leq \frac{2}{\nu}\|f\|_{-1}.
		\end{align}
	\end{lemma}
	
	\begin{remark}
		The stability condition $8\alpha\leq 1$ is sufficient, not necessary. In many applications, the stability condition is far less restrictive.
	\end{remark}

	\begin{proof}
		Choosing $v=u_{k+1}$ in equation (\ref{NewtonO}) gives us 
		\begin{equation}\label{NN11}
			\begin{split}
				&\nu\|\nabla u_{k+1}\|^2=a\left({u}_{k+1},u_{k+1}\right)\\
				&=-b\left({u}_{k},{u}_{k+1},u_{k+1}\right)-b\left({u}_{k+1},{u}_{k},u_{k+1}\right)+\left\langle{f},u_{k+1}\right\rangle+b\left({u}_{k},{u}_{u},u_{k+1}\right)\\
				&=-b\left({u}_{k+1},{u}_{k},u_{k+1}\right)+\left\langle{f},u_{k+1}\right\rangle+b\left({u}_{k},{u}_{k},u_{k+1}\right)\\
				&\leq M\|\nabla u_k\| \|\nabla u_{k+1}\|^2+M\|\nabla u_{k+1}\|\|\nabla u_{k}\|^2+\|f\|_{-1}\|\nabla u_{k+1}\|.
				%	&\leq M\|\nabla u_k\| \|\nabla u_{k+1}\|^2+M\|\nabla u_{k+1}\|\|\nabla w_{k}\|^2+\|f\|_{-1}\|\nabla u_{k+1}\|+C\mu\nu^{-1}\|f\|_{-1}\|\nabla u_{k+1}\|.
			\end{split}
		\end{equation}
		Rearranging (\ref{NN11}) leads to 
		\begin{equation}\label{NN22}
			\begin{split}
				&	(\nu-M\|\nabla u_k\|)\|\nabla u_{k+1}\|\leq M\|\nabla u_{k}\|^2+\|f\|_{-1}.
			\end{split}
		\end{equation}
%		Suppose $\frac{MN_1}{\nu^2}\|f\|_{-1}\leq 1$ and $\|\nabla u_k\|\leq\frac{N_2}{\nu}\|f\|_{-1} $, where $N_1$ and $N_2$ are two positive constants and $N1>N2\geq 1$. Then, based on inequality (\ref{NN22}), we deduce
	{\color{red}	{	Suppose $\frac{MN_1}{\nu^2}\|f\|_{-1}\leq 1$ and $\|\nabla u_k\|\leq\frac{N_2}{\nu}\|f\|_{-1} $ and  $N_1>N_2> 1$,  we wish to find such constants $N_1$ and $N_2$ that satisfy inequality (\ref{NN22}).}} We deduce based on  (\ref{NN22}):
		\begin{equation}\label{NN33}
			\begin{split}
				\nu\left(1-\frac{N_2}{N_1}\right)\|\nabla u_{k+1}\|&\leq\nu\left(1-\frac{MN_2}{\nu^2}\|f\|_{-1}\right)\|\nabla u_{k+1}\|	=\left(\nu-\frac{MN_2}{\nu}\|f\|_{-1}\|\right)\|\nabla u_{k+1}\|\\
				&\leq\left(\nu-M\|\nabla u_k\|\right)\|\nabla u_{k+1}\|\leq M\|\nabla u_{k}\|^2+\|f\|_{-1}\\
				&\leq \frac{MN_2^2}{\nu^2}\|f\|_{-1}^2+\|f\|_{-1}\leq \frac{N_2^2}{N_1}\|f\|_{-1}+\|f\|_{-1},\\
			\end{split}
		\end{equation}
		which is equivalent to 
		\begin{equation}\label{NN44}
			\begin{split}
				\|\nabla u_{k+1}\|
				&\leq \frac{N_1}{N_1-N_2}\left(\frac{N_2^2}{N_1}+1\right)\frac{\|f\|_{-1}}{\nu}.\\
			\end{split}
		\end{equation}
		In order to keep an uniform bound of $\{u_k\}$ for $k=0,1,2,3,....$, we need to impose 
		\begin{equation}\label{NN55}
			\begin{split}
				\frac{N_1}{N_1-N_2}\left(\frac{N_2^2}{N_1}+1\right)\leq N_2,
			\end{split}
		\end{equation}
		which is the same as 
		\begin{equation}\label{NN66}
			\begin{split}
				N_1\geq \frac{2N_2^2}{N_2-1}.
			\end{split}
		\end{equation}
		The constant $N_1$ achieves a minimum $8$ when $N_2=2$ under the constraint $N1>N2\geq 1$, which completes the proof. 
	\end{proof}
	\vspace{2mm}
	{\color{red}	{\begin{remark}
		The choice of $N_1$ and $N_2$ in the above is not unique, and we only used the minimum $N_1$ choice to put the least restriction on $\alpha$.
	\end{remark}}}

%	\begin{remark}
%		The choice of $N_1$ and $N_2$ is not unique. However, to introduce the least restriction on $\alpha$, we try to find $N_1$ and $N_2$ such that
%		\begin{align}
%			(N_1,N_2)=\min_{N_1,N_2\in \mathbb{R}^+}\max\{N_1,N_2\}	~~\text{subject to ~~}	N_1\geq \frac{2N_2^2}{N_2-1},
%		\end{align}
%		which gives $N_1=8$ and $N_2=2$.   
%	\end{remark}
%	

We can now prove Lemma \ref{nc}.
	
	\begin{proof}[Proof of Lemma \ref{nc}]
		Subtracting equation (\ref{NewtonO}) from (\ref{wd}) we have
		\begin{equation}\label{N2}
			\begin{split}
				0&=a(u,v)-a(u_{k+1},v)+b(u,u,v)-b(u_k,u_{k+1},v)-b(u_{k+1},u_{k},v)+b(u_k,u_{k},v)\\
				%	&=a(u-u_{k+1},v)+b(u,u,v)-b(u_k,u,v)+b(u_k,u,v)-b(u_k,u_{k+1},v)+b(u,u_k,v)\\
				%	&~~~~-b(u_{k+1},u_{k},v)-b(u,u_k,v)+b(u_k,u_{k},v)\\
				%	&=a(u-u_{k+1},v)+b(u-u_k,u,v)+b(u_k,u-u_{k+1},v)+b(u-u_{k+1},u_k,v)-b(u-u_k,u_k,v)\\
				&=a(u-u_{k+1},v)+b(u-u_k,u-u_k,v)+b(u_k,u-u_{k+1},v)+b(u-u_{k+1},u_k,v)\\
				&=a(e_{k+1},v)+b(e_k,e_k,v)+b(u_k,e_{k+1},v)+b(e_{k+1},u_k,v)
			\end{split}
		\end{equation}
		Letting $v=e_{k+1}$ in (\ref{N2}), and standard bounds on this trilinear term, (\ref{snewton}), and Young's inequality, we obtain
		\begin{equation}
			\begin{split}
				\nu\|\nabla e_{k+1}\|^2&=b(e_k,e_k,e_{k+1})+b(u_k,e_{k+1},e_{k+1})+b(e_{k+1},u_k,e_{k+1})\\
				&=-b(e_k,e_k,e_{k+1})-b(e_{k+1},u_k,e_{k+1})\\
				&\leq M\|\nabla e_{k}\|^2\|\nabla e_{k+1}\|+M\|\nabla u_k\|\|\nabla e_{k+1}\|^2\\
				&\leq M\|\nabla e_{k}\|^2\|\nabla e_{k+1}\|+\frac{2M}{\nu}\|f\|_{-1}\|\nabla e_{k+1}\|^2,
			\end{split}
		\end{equation}
		which gives us 
		\begin{equation}
			\begin{split}
				\left(\nu-\frac{2M}{\nu}\|f\|_{-1}\right)\|\nabla e_{k+1}\|\leq M\|\nabla e_{k}\|^2,
			\end{split}
		\end{equation}
		Since $8\alpha<1$, we have $\nu\left(1-\frac{2M}{\nu^2}\|f\|_{-1}\right)=\nu\left(1-2\alpha\right)>0$. Once $0\leq\frac{M}{\nu(1-2\alpha)}\|\nabla e_{0}\|<1$, then $\{e_k\},~ k=0,1,2,...$ is a contraction sequence and 
		\begin{align}
			\|\nabla e_{k+1}\|\leq \frac{M}{\nu(1-2\alpha)}\|\nabla e_{k}\|^2=\frac{M}{\beta}\|\nabla e_{k}\|^2.  
		\end{align}
		Thus, $e_k$ converges to $0$ as $k\to +\infty$ quadratically, which completes the proof.  
	\end{proof}

\end{document}